\documentclass[11pt,eqno]{article}
\usepackage{amssymb}
\usepackage{amsmath}
\usepackage{array}
\usepackage{graphicx}
\usepackage{cases}
\usepackage{color}
\usepackage{multirow}
\usepackage{picinpar}
\usepackage{fourier}
\usepackage[colorlinks,
            linkcolor=red,
            anchorcolor=green,
            citecolor=blue
            ]{hyperref}

\begin{document}

\newcommand{\defi}{\stackrel{\Delta}{=}}
\newcommand{\qed}{\hphantom{.}\hfill $\Box$\medbreak}
\newcommand{\A}{{\cal A}}
\newcommand{\B}{{\cal B}}
\newcommand{\U}{{\cal U}}
\newcommand{\G}{{\cal G}}
\newcommand{\cZ}{{\cal Z}}
\newcommand{\proof}{\noindent{\bf Proof \ }}
\newcommand\one{\hbox{1\kern-2.4pt l }}
\newcommand{\Item}{\refstepcounter{Ictr}\item[(\theIctr)]}
\newcommand{\QQ}{\hphantom{MMMMMMM}}

\newtheorem{theorem}{Theorem}[section]
\newtheorem{lemma}{Lemma}[section]
\newtheorem{pro}{Proposition}[section]
\newtheorem{teorem}{Theorem}[section]
\newtheorem{corollary}{Corollary}[section]
\newtheorem{definition}{Definition}[section]
\newtheorem{remark}{Remark}[section]
\newtheorem{assumption}{Assumption}[section]
\newtheorem{example}{Example}[section]
\newenvironment{cproof}
{\begin{proof}
 [Proof.]
 \vspace{-3.2\parsep}}
{\renewcommand{\qed}{\hfill $\Diamond$} \end{proof}}
\newcommand{\erhao}{\fontsize{21pt}{\baselineskip}\selectfont}
\newcommand{\xiaoerhao}{\fontsize{18pt}{\baselineskip}\selectfont}
\newcommand{\sanhao}{\fontsize{15.75pt}{\baselineskip}\selectfont}
\newcommand{\sihao}{\fontsize{14pt}{\baselineskip}\selectfont}
\newcommand{\xiaosihao}{\fontsize{12pt}{\baselineskip}\selectfont}
\newcommand{\wuhao}{\fontsize{10.5pt}{\baselineskip}\selectfont}
\newcommand{\xiaowuhao}{\fontsize{9pt}{\baselineskip}\selectfont}
\newcommand{\liuhao}{\fontsize{7.875pt}{\baselineskip}\selectfont}
\newcommand{\qihao}{\fontsize{5.25pt}{\baselineskip}\selectfont}

\makeatletter
\newcommand{\figcaption}{\def\@captype{figure}\caption}
\newcommand{\tabcaption}{\def\@captype{table}\caption}
\makeatother

\newcounter{Ictr}

\renewcommand{\theequation}{
\arabic{equation}}
\renewcommand{\thefootnote}{\fnsymbol{footnote}}

\def\A{\mathcal{A}}

\def\C{\mathcal{C}}

\def\V{\mathcal{V}}

\def\I{\mathcal{I}}

\def\Y{\mathcal{Y}}

\def\X{\mathcal{X}}

\def\J{\mathcal{J}}

\def\Q{\mathcal{Q}}

\def\W{\mathcal{W}}

\def\S{\mathcal{S}}

\def\T{\mathcal{T}}

\def\L{\mathcal{L}}

\def\M{\mathcal{M}}

\def\N{\mathcal{N}}
\def\R{\mathbb{R}}
\def\H{\mathbb{H}}


\begin{center}
\topskip2cm
\LARGE{\bf Gradient Support Projection Algorithm for Affine Feasibility Problem with Sparsity and Nonnegativity}
\end{center}
\begin{center}
\renewcommand{\thefootnote}{\fnsymbol{footnote}}Lili Pan$^{1,2}$,~Naihua Xiu$^{1}$, ~Shenglong Zhou$^{1}$\footnote{Corresponding author: Shenglong Zhou ( longnan\_zsl@163.com); Other two authors: Lili Pan (panlili1979@163.com), Naihua Xiu (nhxiu@bjtu.edu.cn). Time: June 27, 2014. }\\
{\small
1. Department of Applied Mathematics
Beijing Jiaotong University, Beijing 100044, P. R. China\\
2. Department of Mathematics, Shandong University of Technology,
Zibo 255049, P .R. China}

\end{center}

\vskip12pt

\begin{abstract}

Let $A$ be a real $M\times N$ measurement matrix and $b\in \mathbb{R}^M$ be an observations vector.
The affine feasibility problem with sparsity and nonnegativity ($AFP_{SN}$ for short) is to find a
sparse and nonnegative vector $x\in \mathbb{R}^N$ with $Ax=b$ if such $x$ exists. In this paper, we focus on establishment of optimization approach to solving the $AFP_{SN}$. By discussing tangent cone and normal cone of sparse constraint, we give the first necessary optimality conditions, $\alpha$-Stability, T-Stability and N-Stability, and the second necessary and sufficient optimality conditions for the related minimization problems with the $AFP_{SN}$. By adopting Armijo-type stepsize rule, we present a framework of gradient support projection algorithm for the $AFP_{SN}$ and prove its full convergence when matrix A is s-regular. By doing some numerical experiments, we show the excellent performance of the new algorithm for the $AFP_{SN}$ without and with noise.

\end{abstract}
\section{Introduction }

In this paper, we mainly study an optimization approach to solving the affine feasibility problem
with sparsity and nonnegativity ($AFP_{SN}$) defined by
\begin{equation}
\textrm{Find~the~ vector}~x\in \mathbb{R}^N~ \textrm{with} ~x\geq 0,~ \|x\|_0\leq s~
\textrm{such ~that} ~Ax=b
\end{equation}
if such $x$ exists, where $A\in \mathbb{R}^{M\times N}$, $b\in \mathbb{R}^M$, $s<M<N$ and $\|x\|_0$ is the $l_0$-norm of $x$,
which refers to the number of non-zero elements in the vector $x$.
Vector $x$ is said to be $s$-sparse if $\|x\|_0\leq s$.
For $x=(x_1, \cdots, x_N)^T,y=(y_1, \cdots, y_N)^T\in \mathbb{R}^{N}$, $x\geq y$ stands for $x_i\geq y_i, i=1,2,\cdots,N$.

This is a class of inverse problems and has been popular
for several years due to their applications in signal and image processing \cite{Bruckstein09,Donoho05}, machine learning \cite{He11}
and pattern recognition \cite{Blumensath06}, and so on.  
For example, in many real-world problems the underlying parameters $x$ represent
quantities that can take on only nonnegative values, e.g., pixel intensities, frequency counts. In such cases, sparse affine feasibility problem must include nonnegative constraint on the model parameters $x$.

Usually, the $AFP_{SN}$ is reformulated as the following optimization problem:
\begin{eqnarray}\label{p}
&\min&\frac{1}{2}\|Ax-b\| ^2 \nonumber \\
&s.t.&  \|x\|_0\leq s, x \geq 0.
\end{eqnarray}
Let $S\triangleq \{x\in \mathbb{R}^N|~\|x\|_0\leq s\}$, then the feasible region of (\ref{p}) is denoted as $S \cap \mathbb{R}_+ ^N$; here, $\|\cdot\|$ is $l_2$-norm.

Greedy methods for (\ref{p}) without nonnegativity have recently attracted much attention.
One advantage of greedy methods is that they are generally faster than the relaxation approaches,
and they can also be used to recover signals with more complex structures than sparsity such as tree sparse signals \cite{Braniuk10}.
Another advantage of these methods is that many of them have stable recovery properties under certain
conditions \cite{Cartis13}. A variety of greedy methods have been proposed
to tackle the so-called $l_0$-problem, such as matching pursuit (MP) \cite{Mallat93},
orthogonal MP(OMP) \cite{Davis97}, compressive sampling matching pursuit (CoSaMP) \cite{Needell09} and
Subspace pursuit(SP) \cite{Dai09}. In \cite{Bahmani13}, CoSaMP algorithm was extended to the objective function with arbitrary form.
More recently, iterative hard thresholding algorithm (IHT) was proposed in \cite{Blumensath08}.
Here, Beck et.al \cite{Beck13} showed that the limit points of the algorithm are $L$-stationary points if fixed stepsize $1/L$ is smaller than $\frac{1}{\lambda_{\max}(A^TA)}$.
Blumensath \cite{Blumensath09} proposed an involved line-search method
 -- normalised IHT (NIHT)-- to adaptively select the stepsize per iteration.
 Cartis and Thompson \cite{Cartis13} considered the convergence of IHT and NIHT from the aspect of recovery analysis \cite{Cartis13}.
Foucart \cite{Foucart11} combined IHT and CoSaMP getting hard
thresholding pursuit algorithm (HTP). 
While less effort has been made in sparsity and nonnegativity constraints simultaneously.

In this paper, we adopt a support projection method
to solve this type of NP-hard problem starting from the iterative methods.
Firstly, we study the tangent cone and normal cone of the sparse set under the Bouligand and Clarke concepts respectively.
We propose three kinds of stability for sparsity constrained problems
and analyze the relationship among them, which is $\alpha$-Stability, T-Stability and N-Stability. We
show that $\alpha$-stability is most rigorous than the others.
We also give the second order optimality condition for the same optimality problem under the concept of Clarke tangent cone.
Secondly, we present a gradient support projection algorithm with Armijo-type's stepsize (GSPA) and prove the full convergent properties of
the new algorithm
under the condition that matrix $A$ is $s$-regular.
At last, numerical experiments demonstrate that GSPA performs very steadily whether for
recovery without or with noise and is most time-saving compared with other three methods -- NIHT, CoSaMP and SP.

This paper is organized as follows. Section 2 studies the first and second order optimality conditions for a general sparse optimization model.
Section 3 considers the corresponding results in Section 2 for sparsity and nonnegativity constrained problem (\ref{p}).
Section 4 gives the gradient support projection algorithm with Armijo-type stepsize and proves the convergence.
Section 5 tests the performance of the new method. The last section gives some concluding remarks.

\section{Optimality Conditions for Nonlinear Case}\label{Section2}
In this section, we study the first and second order optimality conditions for the following sparsity constrained nonlinear model:
\begin{eqnarray}\label{f-l-0}
\min~~f(x),~~~~~\textup{s.t.}~~x\in S,\end{eqnarray}
where $f(x):\mathbb{R}^N\rightarrow \mathbb{R}$ is once or twice continuously differentiable.

We first consider the projection on sparse set $S$. For $S\subset \mathbb{R}^N$ being nonempty and closed, we call the mapping $\textmd{P}_S \rightrightarrows S$ the
projector onto $S$ if
$$
\textmd{P}_{S}(x):=\arg \min_{y\in S}\|x-y\|.
$$
As $S$ is nonconvex, the orthogonal projection operator $\textmd{P}_{S}(x)$ is not single-valued. It is well known that the sparse projection $\textmd{P}_{S}(x)$ sets all but $s$ largest (in magnitude) elements of $x$ to zero.
If there is no unique such set, a set can be selected either randomly or according to some predefined ordering.
We define $I_s(x):=\{j_1,j_2,\cdots,j_s\}\subseteq\{1,2,\cdots,N\}$ of indices of
$x$ with $\min_{i\in I_s(x)}|x_i|\geq \max_{i\notin I_s(x)}|x_i|$. Then
\begin{equation*}
\textmd{P}_S(x)=\left\{~y\in\mathbb{R}^{N}\Big| y_i=\left\{\begin{array}{ll}
                                                          x_i,&i\in I_s(x), \\
                                                         0,&i \notin I_s(x).
                                                       \end{array}
\right.\right\}
\end{equation*}


\subsection{Tangent Cone and Normal Cone}
Recalling that for any nonempty set $\Omega\subseteq \mathbb{R}^{N}$, its \emph{Bouligand Tangent Cone} $T^B_{\Omega}(\overline{x})$,  \emph{Clarke Tangent Cone} $T^C_{\Omega}(\overline{x})$ and corresponding \emph{Normal Cones} $N^B_{\Omega}(\overline{x})$ and $N^C_{\Omega}(\overline{x})$ at point $\overline{x}\in\Omega$ are defined as \cite{Rockafellar98}:
\begin{equation}\label{Tangent-Normal-Cone}
\begin{array}{l}
T^B_{\Omega}(\overline{x}):=  \left\{~d\in\mathbb{R}^{N}~\left|\right.~\exists~\{x^k\}\subset \Omega, \underset{k\rightarrow \infty}{\lim}x^k=\overline{x},~\lambda_k\geq0,~k=1,2,\cdots~
~\text{such~that}~\underset{k\rightarrow \infty}{\lim}\lambda_k(x^k-\overline{x})=d~\right\},\\
T^C_{\Omega}(\overline{x}):=  \left\{~d\in\mathbb{R}^{N}~\Big|\begin{array}{r}
                                                        ~\text{For}~\forall~\{x^k\}\subset \Omega,~\forall~\{\lambda_k\}\subset\mathbb{R}_+ ~\text{with}~\underset{k\rightarrow \infty}{\lim}x^k=\overline{x},~\lim_{k\rightarrow\infty}\lambda_k=0, \\
                                                        \exists~\{y^k\}~\text{such~that}~\underset{k\rightarrow \infty}{\lim}y^k=d~\text{and}~x^k+\lambda_{k}y^k\in \Omega,~ k=1,2,\cdots~  \end{array}\right\},\\

N^B_{\Omega}(\overline{x}):= \left\{~d\in\mathbb{R}^{N}~|~\langle d, z\rangle\leq0,~\forall~z\in T^B_{\Omega}(\overline{x})~\right\},\\
N^C_{\Omega}(\overline{x}):= \left\{~d\in\mathbb{R}^{N}~|~\langle d, z\rangle\leq0,~\forall~z\in T^C_{\Omega}(\overline{x})~\right\}.
\end{array}\end{equation}
\begin{theorem}\label{theorem-Btan-Bnor} For any $\overline{x}\in S$ and letting $\Gamma=\textup{supp}(\overline{x})$, the Bouligand tangent cone and corresponding normal cone of ~$S$ at $\overline{x}$ are
\begin{eqnarray}\label{BTangent-Cone-S}
T^B_{S}(\overline{x})&=&\{~d\in\mathbb{R}^{N}~|~\|d\|_0\leq s,~\|\overline{x}+\mu d\|_0\leq s~,\forall~\mu\in\mathbb{R}~\}\\
\label{BTangent-Cone-S1}&=& \bigcup_{\Upsilon}~\textup{span}\left\{~\mathrm{e}_i,~~ i\in\Upsilon\supseteq\Gamma, |\Upsilon|\leq s ~\right\}\\
\label{BNormal-Cone-S}N^B_{S}(\overline{x})&=& \left\{\begin{array}{ll}
                                   \left\{~d\in \mathbb{R}^N|d_i=0,~ i\in\Gamma~\right\}=\textup{span}\left\{~\mathrm{e}_i,~~ i\notin\Gamma~\right\}, & \textup{if}~~|\Gamma|=s\\
                                   \{0\}, & \textup{if}~~|\Gamma|<s
                                 \end{array}\right.\end{eqnarray}
where $\mathrm{e}_i\in\mathbb{R}^{N}$ is a vector whose the $i$th component is one and others are zeros,
$\textup{span}\{\mathrm{e}_i,i\in \Gamma\}$ denotes the subspace of $\mathbb{R}^{N}$ spanned by $\{~\mathrm{e}_i,~i\in \Gamma\}$, and
$\textup{supp}(x)=\{i\in \{1,\cdots,N\}~|~ x_i\neq 0\}$.
\end{theorem}
\begin{proof} It is not difficult to verify that the right hand of (\ref{BTangent-Cone-S}) is equal to (\ref{BTangent-Cone-S1}), and thus we only prove  (\ref{BTangent-Cone-S}). First we prove $T^B_{S}(\overline{x})\subseteq\{~d\in\mathbb{R}^{N}~|~\|d\|_0\leq s,~\|\overline{x}+\mu d\|_0\leq s~,\forall~\mu\in\mathbb{R}~\}$. For $\forall~d\in T^B_{S}(\overline{x})$, there is $\lim_{k\rightarrow\infty}x^k=\overline{x},~x^k\in S,~\lambda_k\geq0$ satisfies $d=\lim_{k\rightarrow\infty}\lambda_k(x^k-\overline{x})$. Since $\lim_{k\rightarrow\infty}x^k=\overline{x}$, there is $k_0$ when $k\geq k_0$, $\Gamma\subseteq\textup{supp}(x^k)$. In addition, $d=\lim_{k\rightarrow\infty}\lambda_k(x^k-\overline{x})$ derives $\textup{supp}(d)\subseteq\textup{supp}(x^k)$, which combining with  $\|x^k\|_0\leq s$ when $k\geq k_0$ and $\Gamma\subseteq\textup{supp}(x^k)$ yields that $\|d\|_0\leq s$ and $\|\overline{x}+\mu d\|_0\leq s~,\forall~\mu\in\mathbb{R}$.
Next we prove $T^B_{S}(\overline{x})\supseteq\{~d\in\mathbb{R}^{N}~|~\|d\|_0\leq s, \|\overline{x}+\mu d\|_0\leq s~,\forall~\mu\in\mathbb{R}~\}$. For $\forall~\|d\|_0\leq s, \|\overline{x}+\mu d\|_0\leq s~,\forall~\mu\in\mathbb{R}$, we take any sequence $\{\lambda_k\}$ such that $\lambda_k>0$ and $\lambda_k\rightarrow+\infty$. Then by defining $\{x^k\}$ with $x^k=\overline{x}+d/\lambda_k$, evidently $x^k\in S$, $\lim_{k\rightarrow\infty}x^k=\overline{x}$, and $d=\lim_{k\rightarrow\infty}\lambda_k(x^k-\overline{x}),$ which implies $d\in T^B_{S}(\overline{x})$.

For (\ref{BNormal-Cone-S}), by the definition of $N^B_{S}(\overline{x})$, we obtain
\begin{eqnarray}\label{Normal-Cone-SS}N^B_{S}(\overline{x})&=&
  \{~d\in\mathbb{R}^{N}~|~\langle d, z\rangle\leq0,~\forall~z\in T^B_{S}(\overline{x})~\}\nonumber\\
  &=&\{~d\in\mathbb{R}^{N}~|~\langle d, z\rangle\leq0,~\|z\|_0\leq s, \|\overline{x}+\mu z\|_0\leq s~,\forall~\mu\in\mathbb{R}~\}.\end{eqnarray}
If $|\Gamma|=s$, it yields $\textup{supp}(z)\subseteq\Gamma$ for any $z\in T^B_{S}(\overline{x})$. Then
\begin{eqnarray*}d\in N^B_{S}(\overline{x})\Longleftrightarrow \langle d, z\rangle\leq0,~\forall~\textup{supp}(z)\subseteq\Gamma
\Longleftrightarrow d_i\left\{ \begin{array}{ll}
                                                                                              =0,&i\in\Gamma, \\
                                                                                              \in\mathbb{R},& i\notin\Gamma.
                                                                                            \end{array}\right.\Longleftrightarrow d\in\textup{span}\left\{~e_i,~~ i\notin\Gamma~\right\}.
~\end{eqnarray*}
If $|\Gamma|<s$, we will prove $ N^B_{S}(\overline{x})=\{0\}$.
Assume $d\in N^B_{S}(\overline{x})$, we take $z=d_{i_0}\mathrm{e}_{i_0}$,$\forall i_0\in \{1,2,\cdots\}$, then $z\in T^B_{S}(\overline{x})$ since $|\Gamma|<s$.
By $\langle d, z\rangle=d_{i_0}^2\leq 0$, we can obtain $d_{i_0}=0$.
The arbitrariness of $i_0$ yields that $d=0$, henceforth $ N^B_{S}(\overline{x})=\{0\}$.\qed\end{proof}
\begin{theorem}\label{lemma-Ctan-Cnor} For any $\overline{x}\in S$ and letting $\Gamma=\textup{supp}(\overline{x})$, then the Clarke tangent cone and corresponding normal cone of $S$ at $\overline{x}$ are
\begin{eqnarray}
T^C_{S}(\overline{x})\label{CTangent-Cone-S}&=&\{~d\in\mathbb{R}^{N}~|~\textup{supp}(d)\subseteq\Gamma~\}
=\textup{span}\left\{~\mathrm{e}_i,~~ i\in\Gamma~\right\}\\
\label{CNormal-Cone-S}N^C_{S}(\overline{x})&=& \textup{span}\left\{~\mathrm{e}_i,~~ i\notin\Gamma~\right\}.\end{eqnarray}
\end{theorem}
\begin{proof}Obviously, $\textup{span}\left\{~\mathrm{e}_i,~~ i\in\Gamma~\right\}=\{~d\in\mathbb{R}^{N}~|~\textup{supp}(d)\subseteq\Gamma~\}$.

We first prove $T^C_{S}(\overline{x})\subseteq\{~d\in\mathbb{R}^{N}~|~\textup{supp}(d)\subseteq\Gamma~\}$. For $\forall~d\in T^C_{S}(\overline{x})$, we have $\forall~\{x^k\}\subset S,~\forall~\{\lambda_k\}\subset\mathbb{R}_+ $ with $\lim_{k\rightarrow\infty}x^k=\overline{x},~\lim_{k\rightarrow\infty}\lambda_k=0,$ there is a sequence $\{y^k\}$ such~that $\lim_{k\rightarrow\infty}y^k=d$ and $x^k+\lambda_{k}y^k\in S,~ k=1,2,\cdots$. Assume that $\textup{supp}(d)\nsubseteq\Gamma$, namely there is an $i_0\in\textup{supp}(d)$ but $i_0\notin\Gamma$. Since $\lim_{k\rightarrow\infty}y^k=d$, it must have $y^{k}_{i_0}\rightarrow d_{i_0}$ which requires $y^{k}_{i_0}\neq 0$ when $k\geq k_0$.
By the arbitrariness of $\{x^k\}$, we take $\{x^k\}\subset S$ such that $\lim_{k\rightarrow\infty}x^k=\overline{x}$, $\textup{supp}(x^k)=\Gamma\cup\Gamma_k$
 with $|\Gamma\cup\Gamma_k|=s$, where $\Gamma_k\subset\{1,2,\cdots,N\}\backslash\Gamma$, and $i_0\notin \Gamma_k$.  Because $\{y^k\}$ is fixed and the arbitrariness of $\{\lambda_k\}$, we can take $\{\lambda_k\}$ which satisfies $\lambda_k<1$ and $\lambda_k/(\min_{i\in\Gamma\cup\Gamma_k,y^k_i\neq 0}|y_i^k|)\rightarrow0$. Now we let
$$\lambda'_k=\min_{i\in\Gamma\cup\Gamma_k,y^k_i\neq 0}\lambda_k\left|\frac{x^k_i}{y^k_i}\right|.$$
Then $\lambda'_k(\neq 0)\rightarrow0$ as $k\rightarrow\infty$. And thus $\forall i\in\Gamma\cup\Gamma_k$, either $|x^k_i+\lambda'_ky_i^k|=|x^k_i|>0$  due to $y^k_i=0$ or
\begin{eqnarray}|x^k_i+\lambda'_ky_i^k|\geq|x^k_i|-\lambda'_k|y_i^k|&=&|x^k_i|-|y_i^k|\min_{i\in\Gamma\cup\Gamma_k,y^k_i\neq 0}\lambda_k\left|\frac{x^k_i}{y^k_i}\right|\geq\left(1-\lambda_k\right)|x^k_i|>0.\nonumber\end{eqnarray}
Moreover, from $i_0\notin\Gamma\cup\Gamma_k$ deriving $x^{k}_{i_0}=0,y^{k}_{i_0}\neq0$, we must have $\|x^k+\lambda'_ky^k\|_0\geq s+1$ for $k\geq k_0$, which is contradicted to $x^k+\lambda'_ky^k\in S$ for any $k=1,2,\cdots$. Therefore $\textup{supp}(d)\subseteq\Gamma$.

 Next we prove $T^C_{S}(\overline{x})\supseteq\{~d\in\mathbb{R}^{N}~|~\textup{supp}(d)\subseteq\Gamma~\}$. For $\forall~d\in\mathbb{R}^{N}$ such that $\textup{supp}(d)\subseteq\Gamma$ and  $\forall~\{x^k\}\subset S,~\forall~\{\lambda_k\}\subset\mathbb{R}_+ $ with $\lim_{k\rightarrow\infty}x^k=\overline{x},~\lim_{k\rightarrow\infty}\lambda_k=0,$ we have $\textup{supp}(d)\subseteq\Gamma\subseteq\textup{supp}(x^k)$ for any $k\geq k_0$. Let
 \begin{equation*}\begin{array}{ll}
                    y^{k}=0, & k=1,2,\cdots,k_0, \\
                    y^{k}=x^{k}-\overline{x}+d, & k=k_0+1,k_0+2,\cdots,
                  \end{array}\end{equation*}
which brings out $x^k+\lambda_ky^k\in S$ for $k=1,2,\cdots$ due to $x^k\in S$. In addition, $\lim_{k\rightarrow\infty}y^k=\lim_{t\rightarrow\infty}x^k-\overline{x}+d=d$. Hence $d\in T^C_{S}(\overline{x})$.

Finally (\ref{CNormal-Cone-S}) holding is obvious. Then the whole proof is completed. \qed\end{proof}
\begin{remark}\label{remark-tan-nor-SN}Clearly for any $\overline{x}\in S$, Bouligand tangent cone $T^B_{S}(\overline{x})$ is closed but non-convex, while Clarke tangent cone $T^C_{S}(\overline{x})$ is closed and convex. In addition, $T^C_{S}(\overline{x})\subseteq T^B_{S}(\overline{x})$.
\end{remark}

\subsection{$\alpha$-Stability, $N$-Stability and $T$-Stability}

When $f(x)$ is continuously differentiable on $\mathbb{R}^N$, we give the definition of three kinds of stability.
\begin{definition}
For real number $\alpha>0$, a vector $x^*\in S$ is called an $\alpha$-stationary point, $N^\sharp$-stationary point and $T^\sharp$-stationary point of $(\ref{f-l-0})$ if it respectively satisfies the relation
\begin{eqnarray}\label{alpha-s-pB}
\alpha-\textup{stationary~point:}~~~~~~~x^*&\in& P_{S}\left(x^*-\alpha \nabla f(x^*)\right),\\
\label{N-s-pB}N^\sharp-\textup{stationary~point:}~~~~~~~~0~~&\in& \nabla f(x^*)+N^\sharp_{S}(x^*),\\
\label{T-s-pB}T^\sharp-\textup{stationary~point:}~~~~~~~~0~~&=&\|\nabla^{\sharp}_S f(x^*)\|,\end{eqnarray}
where $\nabla^{\sharp}_S f(x^*)=\arg\min\{~\|x+\nabla f(x^*)\|~|~x\in T^\sharp_S(x^*)~\}$,
$\sharp\in\{B,C\}$ stands for the sense of Bouligand tangent cone or Clarke tangent cone.
\end{definition}

\begin{theorem}\label{Btheorem-alpha-N-T}Under the concept of Bouligand tangent cone, we consider model $(\ref{f-l-0})$. For $\alpha>0$,\\
if the vector $x^*\in S$ satisfies $\|x^*\|_0=s$, then
$$ \alpha-\text{stationary~point}~~~\Longrightarrow~~~N^B-\text{stationary point}~~~\Longleftrightarrow~~~T^B-\text{stationary point} ;$$
if the vector $x^*\in S$ satisfies $\|x^*\|_0<s$, then
 $$\alpha-\text{stationary~point}~~\Longleftrightarrow~~N^B-\text{stationary~point}~~\Longleftrightarrow~~T^B-\text{stationary~point}~~\Longleftrightarrow~~\nabla f(x^*)=0.$$
\end{theorem}
\begin{proof}Denote $\Gamma=\textup{supp}(x^*)$. If $x^*$ is an $\alpha$-stationary point of model $(\ref{f-l-0})$, then from Lemma 2.2 in \cite{Beck13}, it holds
\begin{eqnarray}\label{alpha-eqaul}
x^*\in P_{S}\left(x^*-\alpha \nabla f(x^*)\right)~\Longleftrightarrow~\left|(\nabla f(x^*))_i\right|~\left\{\begin{array}{ll}
                      =0, &  i\in \Gamma\\
                      \leq \frac{1}{\alpha} M_s(|x^*|), & i\notin \Gamma,
                    \end{array}
\right.
\end{eqnarray}
for any $\alpha>0$, where $M_s(|x^*|)$ is the $s$th largest element of $|x^*|$.

Case 1. First we consider the case $\|x^*\|_0= s$. Under such circumstance, if $x^*$ is an $N^B$-stationary point of model $(\ref{f-l-0})$,
then by (\ref{BNormal-Cone-S}) in Theorem \ref{theorem-Btan-Bnor}, we have
\begin{eqnarray}\label{BN-eqaul}
-\nabla f(x^*)\in N^B_S(x^*)~\Longleftrightarrow~(\nabla f(x^*))_i~\left\{\begin{array}{ll}
                      =0, &  i\in \Gamma\\
                      \in \mathbb{R}, & i\notin \Gamma,
                    \end{array}
\right.
\end{eqnarray}

Moreover, $\|x^*\|_0= s$  produces
\begin{eqnarray}\nabla^{B}_S f(x^*)&=&\textup{argmin}\{~\|d+\nabla f(x^*)\|~|~d\in T^B_S(x^*)~\}\nonumber\\
&=&\textup{argmin}\{~\|d+\nabla f(x^*)\|~|~\|d\|_0\leq s,~\|x^*+\mu d\|_0\leq s~,\forall~\mu\in\mathbb{R}~\}\nonumber\\
\label{BT-T}&=&\textup{argmin}\{~\|d+\nabla f(x^*)\|~|~\textup{supp}(d)\subseteq\Gamma~\},\end{eqnarray}
where the third equality holds due to $\|x^*\|_0= s$. (\ref{BT-T}) is equivalent to
\begin{eqnarray*}\label{BT-T0}(\nabla^{B}_S f(x^*))_i=\left\{\begin{array}{ll}
                                          -(\nabla f(x^*))_i,& i\in \Gamma, \\
                                          0, & i\notin \Gamma.
                                        \end{array}
\right. \end{eqnarray*}
Therefore, if $x^*$ is an $T^B$-stationary point of model $(\ref{f-l-0})$, then from above
\begin{eqnarray}\label{BT-eqaul}
\nabla^{B}_{S} f(x^*)=0~\Longleftrightarrow~(\nabla f(x^*))_i~\left\{\begin{array}{ll}
                      =0, &  i\in \Gamma\\
                      \in \mathbb{R}, & i\notin \Gamma,
                    \end{array}
\right.
\end{eqnarray}
Henceforth, from (\ref{alpha-eqaul}), (\ref{BN-eqaul}) and (\ref{BT-eqaul}), one can easily check that when $\|x^*\|_0= s$
$$ \alpha-\text{stationary~point}~~~\Longrightarrow~~~N^B-\text{stationary point}~~~\Longleftrightarrow~~~T^B-\text{stationary point}.$$

Case 2. Now we consider the case $\|x^*\|_0<s$. Under such circumstance, $M_s(|x^*|)=0$, and thus if $x^*$ is an $\alpha$-stationary point of model $(\ref{f-l-0})$, then from (\ref{alpha-eqaul}), it holds
\begin{eqnarray}\label{alpha-eqaul-1}
x^*\in P_{S}\left(x^*-\alpha \nabla f(x^*)\right)~\Longleftrightarrow~\nabla f(x^*)=0.\end{eqnarray}

Then  when $\|x^*\|_0<s$, $N^B_{S}(x^*)=\{0\}$ from (\ref{BNormal-Cone-S}), which implies $\nabla f(x^*)=0$. Therefore, if $x^*$ is an $N^B$-stationary point of model $(\ref{f-l-0})$, then
\begin{eqnarray}\label{BN-eqaul-1}
0\in\nabla f(x^*)+N^B_{S}(x^*)~\Longleftrightarrow~\nabla f(x^*)=0.\end{eqnarray}

Finally, we prove $\nabla f(x^*)=0\Longleftrightarrow\nabla^{B}_S f(x^*)=0$ when $\|x^*\|_0<s$. On one hand, if $\nabla f(x^*)=0$, then
\begin{eqnarray}\nabla^{B}_S f(x^*)&=&\textup{argmin}\{~\|d+\nabla f(x^*)\|~|~\|d\|_0\leq s,~\|x^*+\mu d\|_0\leq s~,\forall~\mu\in\mathbb{R}~\}\nonumber\\
&=&\textup{argmin}\{~\|d\|~|~\|d\|_0\leq s,~\|x^*+\mu d\|_0\leq s~,\forall~\mu\in\mathbb{R}~\}=0.\nonumber\end{eqnarray}
On the other hand, if $\nabla^{B}_S f(x^*)=0$, then
\begin{eqnarray}0=\nabla^{B}_S f(x^*)=\textup{argmin}\{~\|d+\nabla f(x^*)\|~|~\|d\|_0\leq s,~\|x^*+\mu d\|_0\leq s~,\forall~\mu\in\mathbb{R}~\}\nonumber\end{eqnarray}
leads to $\|\nabla f(x^*)\|\leq\|d+\nabla f(x^*)\|$ for any $\|d\|_0\leq s,~\|x^*+\mu d\|_0\leq s~,\forall~\mu\in\mathbb{R}$. Particular, for $\forall~i_0\in\{1,2,\cdots,N\}$, we take $d$ with $\textup{supp}(d)=\{i_0\}$. Apparently, $\|x^*+\mu d\|_0\leq s~,\forall~\mu\in\mathbb{R}$ owing to $\|x^*\|_0<s$. Then by valuing $d_{i_0}=-(\nabla f(x^*))_{i_0}$ and $d_{i}=0,~i\neq i_0$, we immediately get $(\nabla f(x^*))_{i_0}=0$ because of $\|\nabla f(x^*)\|\leq\|\nabla f(x^*)-(\nabla f(x^*))_{i_0}\|$. Then by the arbitrariness of $i_0$, it holds $\nabla f(x^*)=0$.
Therefore, if $x^*$ is an $T^B$-stationary point of model $(\ref{f-l-0})$, then
\begin{eqnarray}\label{BT-eqaul-1}
\nabla^{B}_S f(x^*)=0~\Longleftrightarrow~\nabla f(x^*)=0.\end{eqnarray}
Henceforth, from (\ref{alpha-eqaul-1}), (\ref{BN-eqaul-1}) and (\ref{BT-eqaul-1}), one can easily check that when $\|x^*\|_0<s$
$$ \alpha-\text{stationary~point}~~\Longleftrightarrow~~N^B-\text{stationary point}~~\Longleftrightarrow~~T^B-\text{stationary point}~~\Longleftrightarrow~\nabla f(x^*)=0.$$
Overall, the whole proof is finished.\qed\end{proof}

Based on the proof of Theorem \ref{Btheorem-alpha-N-T}, we use the following table to illustrate the relationship among these three stationary points under the concept of Bouligand tangent cone.

 \tabcaption{The relationship among these three kinds ot stationary points.\label{relationship-B}}
{\renewcommand\baselinestretch{1.5}\selectfont
{\centering\begin{tabular}{rc|l|c}\hline
&&~~~~~~~~~~~~~~~~~~~~~~~~$\|x^*\|_0=s$&$\|x^*\|_0<s$\\\hline
$\alpha$ -- stationary~point&\multirow{2}{*}{$\Longleftrightarrow$}&\multirow{2}{*}{$\left|(\nabla f(x^*))_i\right|~\left\{\begin{array}{ll}
                      =0, &  i\in \Gamma\\
                      \leq \frac{1}{\alpha} M_s(|x^*|), & i\notin \Gamma,
                    \end{array}
\right.$}&\multirow{2}{*}{$\nabla f(x^*)=0$}\\
$x^*\in P_{S}\left(x^*-\alpha \nabla f(x^*)\right)$&&&\\\hline
$N^B$ -- stationary~point&\multirow{2}{*}{$\Longleftrightarrow$}&\multirow{2}{*}{$~~(\nabla f(x^*))_i~~~\left\{\begin{array}{ll}
                      =0, &  i\in \Gamma\\
                      \in \mathbb{R}, & i\notin \Gamma,
                    \end{array}
\right.$}&\multirow{2}{*}{$\nabla f(x^*)=0$}\\
$-\nabla f(x^*)\in N^B_S(x^*)$&&&\\\hline
$T^B$ -- stationary~point&\multirow{2}{*}{$\Longleftrightarrow$}&\multirow{2}{*}{$~~(\nabla f(x^*))_i~~~\left\{\begin{array}{ll}
                      =0, &  i\in \Gamma\\
                      \in \mathbb{R}, & i\notin \Gamma,
                    \end{array}
\right.$}&\multirow{2}{*}{$\nabla f(x^*)=0$}\\
$\nabla^{B}_{S} f(x^*)=0$&&&\\\hline
\end{tabular}\par}}
\vspace{6mm}

\begin{theorem}\label{Ctheorem-alpha-N-T}Under the concept of Clarke tangent cone, we consider model $(\ref{f-l-0})$. For $\alpha>0$, if $x^*\in S$ then
$$ \alpha-\text{stationary~point}~~~\Longrightarrow~~~N^C-\text{stationary point}~~~\Longleftrightarrow~~~T^C-\text{stationary point}.$$
\end{theorem}
\begin{proof}
Denote $\Gamma=\textup{supp}(x^*)$. If $x^*$ is an $\alpha$-stationary point of model $(\ref{f-l-0})$, for any $\alpha>0$, we have (\ref{alpha-eqaul})

If $x^*$ is an $N^C$-stationary point of model $(\ref{f-l-0})$, then by (\ref{CNormal-Cone-S}), we have
\begin{eqnarray}\label{CN-eqaul}
-\nabla f(x^*)\in N^C_S(x^*)~\Longleftrightarrow~(\nabla f(x^*))_i~\left\{\begin{array}{ll}
                      =0, &  i\in \Gamma\\
                      \in \mathbb{R}, & i\notin \Gamma,
                    \end{array}
\right.
\end{eqnarray}

Moreover, by (\ref{CTangent-Cone-S}), it follows
\begin{eqnarray*}\nabla^{C}_S f(x^*)&=&\textup{argmin}\{~\|d+\nabla f(x^*)\|~|~d\in T^C_S(x^*)~\}=\textup{argmin}\{~\|d+\nabla f(x^*)\|~|~\textup{supp}(d)\subseteq\Gamma~\},\end{eqnarray*}
which is equivalent to
\begin{eqnarray*}\label{CT-T0}(\nabla^{C}_S f(x^*))_i=\left\{\begin{array}{ll}
                                          -(\nabla f(x^*))_i,& i\in \Gamma, \\
                                          0, & i\notin \Gamma.
                                        \end{array}
\right. \end{eqnarray*}
Therefore, if $x^*$ is an $T^C$-stationary point of model $(\ref{f-l-0})$, then from above
\begin{eqnarray}\label{CT-eqaul}
\nabla^{C}_{S} f(x^*)=0~\Longleftrightarrow~(\nabla f(x^*))_i~\left\{\begin{array}{ll}
                      =0, &  i\in \Gamma\\
                      \in \mathbb{R}, & i\notin \Gamma,
                    \end{array}
\right.
\end{eqnarray}
Henceforth, from (\ref{alpha-eqaul}), (\ref{CN-eqaul}) and (\ref{CT-eqaul}), one can easily check
$$ \alpha-\text{stationary~point}~~~\Longrightarrow~~~N^C-\text{stationary point}~~~\Longleftrightarrow~~~T^C-\text{stationary point}.$$
Overall, the whole proof is finished.\qed\end{proof}

Based on the proof of Theorem \ref{Ctheorem-alpha-N-T}, we use the following table to illustrate the relationship among these three stationary points under the concept of Clarke tangent cone.

 \tabcaption{The relationship among these three kinds of stationary points.\label{relationship-C}}
{\renewcommand\baselinestretch{1.5}\selectfont
{\centering\begin{tabular}{rc|l|c}\hline
&&~~~~~~~~~~~~~~~~~~~~~~~~$\|x^*\|_0=s$&$\|x^*\|_0<s$\\\hline
$\alpha$ -- stationary~point&\multirow{2}{*}{$\Longleftrightarrow$}&\multirow{2}{*}{$\left|(\nabla f(x^*))_i\right|~\left\{\begin{array}{ll}
                      =0, &  i\in \Gamma\\
                      \leq \frac{1}{\alpha} M_s(|x^*|), & i\notin \Gamma,
                    \end{array}
\right.$}&\multirow{2}{*}{$\nabla f(x^*)=0$}\\
$x^*\in P_{S}\left(x^*-\alpha \nabla f(x^*)\right)$&&&\\\hline
$N^C$ -- stationary~point&\multirow{2}{*}{$\Longleftrightarrow$}&\multirow{2}{*}{$~~(\nabla f(x^*))_i~~~\left\{\begin{array}{ll}
                      =0, &  i\in \Gamma\\
                      \in \mathbb{R}, & i\notin \Gamma,
                    \end{array}
\right.$}&\multirow{2}{*}{$(\nabla f(x^*))_i~\left\{\begin{array}{ll}
                      =0, &  i\in \Gamma\\
                      \in \mathbb{R}, & i\notin \Gamma,
                    \end{array}
\right.$}\\
$-\nabla f(x^*)\in N^C_S(x^*)$&&&\\\hline
$T^C$ -- stationary~point&\multirow{2}{*}{$\Longleftrightarrow$}&\multirow{2}{*}{$~~(\nabla f(x^*))_i~~~\left\{\begin{array}{ll}
                      =0, &  i\in \Gamma\\
                      \in \mathbb{R}, & i\notin \Gamma,
                    \end{array}
\right.$}&\multirow{2}{*}{$(\nabla f(x^*))_i~\left\{\begin{array}{ll}
                      =0, &  i\in \Gamma\\
                      \in \mathbb{R}, & i\notin \Gamma,
                    \end{array}
\right.$}\\
$\nabla^{C}_{S} f(x^*)=0$&&&\\\hline
\end{tabular}\par}}
\vspace{6mm}
\subsection{Second Order Optimality Conditions}
In this subsection, we study the second order necessary and sufficient optimality of model $(\ref{f-l-0})$ if $f(x)$ is twice continuously differentiable on $\mathbb{R}^N$ and satisfies the following assumption.
\begin{assumption}\label{ass0}
The gradient of the objective function $f(x)$ is Lipschitz with constant $L_f$ over $\mathbb{R}^N$:
\begin{eqnarray}\label{lip}\|\nabla f(x)-\nabla f(y)\|\leq L_f\|x-y\|,~~~\forall~x, y\in\mathbb{R}^N.\end{eqnarray}
\end{assumption}

\begin{theorem}[Second Order Necessary Optimality]\label{B-Necessary-10}If $x^{*}\in S$ is the optimal solution of $(\ref{f-l-0})$ , then for $0<\alpha<\frac{1}{L_f}$, $x^{*}$ is also the $\alpha$-stationary point, moreover,
\begin{eqnarray}\label{B-Necessary-1}
d^{\top}\nabla^{2}f(x^*)d\geq 0,~~~~\forall~d\in T^C_S(x^*).\end{eqnarray}
where $\nabla^{2}f(x^*)$ is the Hessian matrix of $f$ at $x^*$.
\end{theorem}
\begin{proof}Since $x^{*}$ is the optimal solution of $(\ref{f-l-0})$,
it must be an $\alpha$-stationary point of model $(\ref{f-l-0})$ for $0<\alpha<\frac{1}{L_f}$~(Theorem 2.2 in \cite{Beck13}).
By (\ref{alpha-eqaul}) and (\ref{CTangent-Cone-S}), one can easily verify that
$$d^{\top}\nabla f(x^*)=0,~~~~\forall~d\in T^C_S(x^*).$$
Moreover, for any $\tau>0$ and $d\in T^C_S(x^*)$, by the optimality of $x^*$ and equality above, we have
\begin{eqnarray}\label{semi}
0&\leq& f(x^*+\tau d)-f(x^*)\nonumber\\
&=&f(x^*)+\tau d^{\top}\nabla f(x^*)+\frac{\tau^{2}}{2}d^{\top}\nabla^{2}f(x^*)d+o(\|d\|^2)-f(x^*)\nonumber\\
&=&\frac{\tau^{2}}{2}d^{\top}\nabla^{2}f(x^*)d+o(\|d\|^2),\nonumber\end{eqnarray}
which implies that
$$d^{\top}\nabla^{2}f(x^*)d\geq 0,~~~~\forall~d\in T^C_S(x^*).$$
The desired result is acquired.\qed\end{proof}

\begin{theorem}[Second Order Sufficient Optimality]\label{B-Sufficient-10}If $x^{*}\in S$ is an $\alpha$-stationary point of $(\ref{f-l-0})$ and satisfies
\begin{eqnarray}\label{B-Sufficient-1}
d^{\top}\nabla^{2}f(x^*)d>0,~~~~\forall~d\in T^C_S(x^*),\end{eqnarray}
then $x^{*}$ is the strictly locally optimal solution of $(\ref{f-l-0})$. Moreover, there is a $\gamma>0$ and $\delta>0$, when any $x\in B(x^*, \delta)\cap S$, it holds
\begin{eqnarray}\label{B-Sufficient-2} f(x)\geq f(x^*)+\gamma\|x-x^*\|^{2}.\end{eqnarray}
 \end{theorem}
\begin{proof}We only prove the second conclusion. From Table \ref{relationship-C}, one can easily check
\begin{eqnarray}\label{ort}
d^{\top}\nabla f(x^*)=0,~~\forall~d\in T^C_S(x^*),\end{eqnarray}
By assuming the conclusion does not hold, there must be a sequence $\{x^k\}$ with
$$\lim_{k\rightarrow\infty}x^k=x^*~~~\textup{and}~~~\textup{supp}(x^k)=\textup{supp}(x^*)$$
such that
\begin{eqnarray}\label{B-Sufficient-con} f(x^{k})- f(x^*)\leq\frac{1} {k}\|x^{k}-x^*\|^{2}.\end{eqnarray}
Denote $d^{k}=\frac{x^{k}-x^*}{\|x^{k}-x^*\|}$. Due to $\|\frac{x^{k}-x^*}{\|x^{k}-x^*\|}\|=1$, there exists a convergent subsequence,
without loss of generality, assuming  $d^{k}\rightarrow \bar{d}$.
$d^{k}\in T^C_S(x^*)$ and $\bar{d}\in T^C_S(x^*)$ due to $\textup{supp}(x^k)=\textup{supp}(x^*)$, which means $d^{k\top}\nabla f(x^*)=0$ by (\ref{ort}). From (\ref{B-Sufficient-con}), we have
\begin{eqnarray}\frac{1}{k}&\geq& \frac{1}{\|x^{k}-x^*\|^{2}}\left( f(x^{k})- f(x^*)\right)\nonumber\\
&=&\frac{1}{\|x^{k}-x^*\|^{2}}\left((x^{k}-x^*)^\top\nabla f(x^{k})+\frac{1}{2}(x^{k}-x^*)^\top \nabla^2f(x^*)(x^{k}-x^*)+o(\|x^{k}-x^*\|^{2})\right)\nonumber\\
&=&d^{k\top} \nabla^2f(x^*)d^{k}+ \frac{1}{\|x^{k}-x^*\|}d^{k\top} \nabla f(x^*)+o(1)\nonumber\\
\label{B-Sufficient-con1}&=&d^{k\top} \nabla^2f(x^*)d^{k}+o(1).\end{eqnarray}
Then take the limit of both side of (\ref{B-Sufficient-con1}), we obtain
$$0=\lim_{k\rightarrow\infty}\frac{1}{k}\geq\lim_{k\rightarrow\infty}\left(d^{k\top} \nabla^2f(x^*)d^{k}+o(1)\right)=\bar{d}^{\top} \nabla^2f(x^*)\bar{d}>0, ~~
\bar{d}\in T^C_S(x^*),$$
which is contradicted. Therefore the conclusion does hold.\qed\end{proof}

\section{Optimality Conditions for Model (\ref{p})}\label{Section3}

In this section, we mainly aim at specifying the results in Section \ref{Section2} to the model (\ref{p}). For notational simplicity, we hereafter denote $r(x)\triangleq \frac{1}{2}\|Ax-b\|^2$. First, we define the projection on $S\cap\mathbb{R}_+^{N}$ named \textbf{support projection}, which has an explicit expression.
\begin{pro}\label{prop1}
$\textmd{P}_{S\cap\mathbb{R}_+^{N}}(x)= \textmd{P}_S\cdot\textmd{P}_{\mathbb{R}_+^{N}}(x)$.
\end{pro}
\begin{proof}
Denote $I_+(x)=\{i|x_i>0\}$,$I_0(x)=\{i|x_i=0\}$,$I_-(x)=\{i|x_i<0\}$,
let $y\in\textmd{P}_{S\cap\mathbb{R}_+^{N}}(x)$.
For $i\in I_0(x)\cup I_-(x)$, it is easy to see $y_i=0$.
There are two cases:\\
Case 1, $|I_+(x)|\leq s$, then $y=\textmd{P}_{\mathbb{R}_+^{N}}(x)=\textmd{P}_S\cdot\textmd{P}_{\mathbb{R}_+^{N}}(x)$.\\
Case 2, $|I_+(x)|> s$, we should choose no more than $s$ coordinates from $I_+(x)$ to minimize $\|x-y\|$.
For $i,j\in I_+(x)$ and $x_i>x_j$, $$(x_i-x_i)^2+(x_j-x_j)^2<(x_i-x_i)^2+(x_j-0)^2<(x_i-0)^2+(x_j-x_j)^2<(x_i-0)^2+(x_j-0)^2.$$
Then the projection on $S\cap\mathbb{R}_+^{N}$ sets all but $s$ largest elements of $\textmd{P}_{\mathbb{R}_+^{N}}(x)$ to zero,
which is $\textmd{P}_S\cdot\textmd{P}_{\mathbb{R}_+^{N}}(x)$.
\qed
\end{proof}

Notice that the order of projections can't be changed. For example
$x=(-2,1)^T$, $s=1$. $\textmd{P}_{S\cap\mathbb{R}_+^{2}}(x)= \textmd{P}_S\cdot\textmd{P}_{\mathbb{R}_+^{2}}(x)=(0,1)^T$,
while $\textmd{P}_{\mathbb{R}_+^{2}}\cdot\textmd{P}_S(x)=(0,0)^T$.

The direct result of Theorem \ref{theorem-Btan-Bnor} and \ref{lemma-Ctan-Cnor} is the following theorem.
\begin{theorem}\label{pro-tan-nor-SN} For any $\overline{x}\in S\cap\mathbb{R}^{N}_+$, by denoting~~$\mathbb{R}^{N}_+(\overline{x}):=\{~x\in\mathbb{R}^{N}~|~x_i\geq0,i\notin\Gamma~\}$, it follows
\begin{eqnarray}
\label{T-N-Cone-SN-T}
&&T^B_{S\cap\mathbb{R}^{N}_+}(\overline{x})=T^B_{S}(\overline{x})\cap\mathbb{R}^{N}_+(\overline{x}),
~~~~N^B_{S\cap\mathbb{R}^{N}_+}(\overline{x})=T^B_{S}(\overline{x})\cap(-\mathbb{R}^{N}_+(\overline{x}))\\
\label{T-N-Cone-SN-N}&&T^C_{S\cap\mathbb{R}^{N}_+}(\overline{x})= T^C_{S}(\overline{x}),~~~~~~~~~~~~~~~~~~~~~N^C_{S\cap\mathbb{R}^{N}_+}(\overline{x})= N^C_{S}(\overline{x}).\end{eqnarray}
\end{theorem}

For model $(\ref{p})$, we have the corresponding definition of $\alpha$-stationary point, $N^\sharp$-stationary point and $T^\sharp$-stationary point by substituting
 $S\cap \mathbb{R}_+ ^N$ for $S$, where $\sharp\in\{B,C\}$ stands for the sense of Bouligand tangent cone or Clarke tangent cone.
In order to facilitate the discussion next, we describe a more explicit representation of $\alpha$-stationary point, that is
\begin{equation}\label{NC}
x^* \in P_{S\cap \mathbb{R}^N_+}\left(x^*-\alpha \nabla r(x^*)\right).
\end{equation}

\begin{theorem}\label{th2.1}
For any $\alpha>0$, a vector $x^*\in S\cap \mathbb{R}_+ ^N$ is $\alpha$-stationary point of $(\ref{p})$ if and only if
\begin{eqnarray}\label{NCL}
\nabla_i r(x^*)\left\{
             \begin{array}{lll}
              =0, & \mbox{if}~~ i\in \mathrm{supp}(x^*),\\
             \geq 0 , \mathrm{or} \in[-\frac{1}{\alpha }M_s (x^*),0], &\mbox{if}~~ i \notin \mathrm{supp}(x^*),
             \end{array}
        \right.
\end{eqnarray}
\end{theorem}
\begin{proof}
Suppose $(\ref{NC})$ is satisfied for $x^* $.
If $i\in \mathrm{supp}(x^*)$, then $x^*_i =x^*_i -\alpha \nabla_i r(x^*)$, so that $\nabla_i r(x^*)=0$;
If $i\notin \mathrm{supp}(x^*)$, there are two cases: either $x^*_i -\alpha \nabla_i r(x^*)\leq 0$,
that is $\nabla_i r(x^*)\geq x_i= 0$, or $0\leq x^*_i -\alpha \nabla_i r(x^*)\leq M_s (x^*)$,
that is $-\frac{1}{\alpha}M_s (x^*)\leq  \nabla_i r(x^*)\leq 0$.

On the contrary, assume $(\ref{NCL})$ holds.
If $\|x^*\|_0 <s$, we get $M_s (x^*)=0$, then for $i\in \mathrm{supp}(x^*)$, $\nabla_i r(x^*)=0$, then $x^*-\alpha\nabla_i r(x^*)=x_i^*$
or for $i\notin \mathrm{supp}(x^*)$, $x_i^*-\alpha\nabla_i r(x^*)\leq 0$,
therefore, (\ref{NC}) holds.
If $\|x^*\|_0 =s$, that is $M_s (x^*)>0$. By (\ref{NCL}),
for $i\in \mathrm{supp}(x^*)$, $\alpha\nabla_i r(x^*)= 0$;
for $i\notin \mathrm{supp}(x^*)$,  $x^*-\alpha\nabla_i r(x^*)\leq 0$ or
 $0\leq x_i^*-\alpha \nabla_i r(x^*)\leq M_s (x^*)$,
 so that (\ref{NC}) holds as well.\qed
\end{proof}


One can easily check that when $\|\overline{x}\|_0=s$ it holds  $T^B_{S\cap\mathbb{R}^{N}_+}(\overline{x})=T^B_{S}(\overline{x})$. Therefore the corresponding theorem is derived by Theorems \ref{Btheorem-alpha-N-T}, \ref{Ctheorem-alpha-N-T},
 \ref{th2.1} and Corollary \ref{pro-tan-nor-SN} directly.
\begin{theorem}\label{Btheorem-alpha-N-T-N}For the model $(\ref{p})$ and any $\alpha>0$.\\
\emph{\textbf{A})} Under the concept of Bouligand tangent cone, if~~$\|x^*\|_0=s, x^*\geq 0$, then
$$ \alpha-\text{stationary~point}~~~\Longrightarrow~~~N^B-\text{stationary point}~~~\Longleftrightarrow~~~T^B-\text{stationary point}.$$
\emph{\textbf{B})} Under the concept of Clarke tangent cone, if~~$\|x^*\|_0\leq s, x^*\geq 0$, then
$$ \alpha-\text{stationary~point}~~~\Longrightarrow~~~N^C-\text{stationary point}~~~\Longleftrightarrow~~~T^C-\text{stationary point}.$$
\end{theorem}

Combining Theorems \ref{B-Necessary-10} and \ref{B-Sufficient-10}, we derive the following second order optimality result.

\begin{theorem}[Second Order Optimality]\label{B-Sufficient-N}
If $x^{*}\in S\cap \mathbb{R}^N_+$ is the optimal solution of $(\ref{p})$, then for $0<\alpha<\frac{1}{\lambda_{\max}(A^TA)}$, $x^{*}$ is also the $\alpha$-stationary point of $(\ref{p})$, and moreover,
\begin{eqnarray}\label{B-Necessary-1-N}
d^{\top}A^TAd\geq 0,~~~~\forall~d\in T^C_S(x^*).\end{eqnarray}
On the contrary, if $x^{*}\in S\cap\mathbb{R}^{N}_{+}$ is an $\alpha$-stationary point of $(\ref{p})$ and satisfies
\begin{eqnarray}\label{B-Sufficient-1-N}
d^{\top}A^TAd>0,~~~~\forall~d\in T^C_S(x^*),\end{eqnarray}
then $x^{*}$ is the strictly locally optimal solution of $(\ref{p})$. Moreover, there is a $\gamma>0$ and $\delta>0$, when any $x\in B(x^*, \delta)\cap S\cap\mathbb{R}^{N}_{+}$, it holds
\begin{eqnarray}\label{B-Sufficient-2-N} r(x)\geq r(x^*)+\gamma\|x-x^*\|^{2}.\end{eqnarray}
 \end{theorem}

\section{Gradient Support Projection Algorithm}
We now develop the gradient support projection algorithm with Armijo-type stepsize rule which is shortly denoted as GSPA. For simplicity, we utilize $L_r:=\lambda_{\max}(A^TA)$ to denote the Lipschtiz constant of $\nabla r(x)$.

\begin{table}[h]
 \caption{The framework of GSPA.\label{framework-GSPA}}\vspace{2mm}
\renewcommand{\arraystretch}{1.2}\addtolength{\tabcolsep}{0pt}
\centering
\begin{tabular}{l}\hline
\textbf{Step 0}~~Initialize $x^0=0$, $\Gamma^0=\mathrm{supp}(\textmd{P}_{S\cap \mathbb{R}_+^N}(A^Tb))$, $0<\alpha_0<\frac{1}{L_r}$, $0<\sigma\leq\frac{1}{4L_r}$, $0<\beta<1, \epsilon>0$. Set $k\Leftarrow0$;   \\
\textbf{Step 1}~~Compute~~~~~~~~~~~~$\tilde{x}^{k+1}=\textmd{P}_{S\cap \mathbb{R}_+^N}\big(x^k-\alpha_0 \nabla r(x^k)\big);$\\
\textbf{Step 2}~~If $\mathrm{supp}(\tilde{x}^{k+1})=\Gamma^k$, then $x^{k+1}=\tilde{x}^{k+1},\Gamma^{k+1}=\mathrm{supp}(x^{k+1})$;\\
\hspace{1.2cm}Else compute $x^{k+1}=\textmd{P}_{S\cap \mathbb{R}_+^N}\big(x^k-\alpha_k \nabla r(x^k)\big)$, $\Gamma^{k+1}=\mathrm{supp}(x^{k+1})$, where $\alpha_k=\alpha_0\beta^{m_k}$ and $m_k$\\
\hspace{1.2cm}is the smallest positive integer $m$ such that \\
\hspace{3.5cm}$r(x^{k}(\alpha_0\beta^m))\leq r(x^k)-\frac{\sigma}{2} \frac{\|x^{k}(\alpha_0\beta^m)-x^k\|^2}{(\alpha_0\beta^{m})^2},$\\
\hspace{1.2cm}here $x^{k}(\alpha)=\textmd{P}_{S\cap \mathbb{R}_+^N}(x^k-\alpha \nabla r(x^k)) $;\\
\textbf{Step 3}~~If $\|x^{k+1}-x^{k}\|\leq \epsilon$, then stop; Otherwise $k \Leftarrow k+1$, go to \textbf{Step 1}.\\
\hline
\end{tabular}\end{table}

\noindent\textbf{Remark~}Compared with IHT in \cite{Beck13}, we mainly add Armijo-type stepsize rule in Step 2, which is well defined by Lemma \ref{le3.1}.
\begin{lemma}\label{le3.1}
Let $\left\{x^k\right\}$ be the iterative point in Step 2 in GSPA. Then
 \begin{equation}\label{objective-descent-12}
r(x^{k}(\alpha))\leq\left\{\begin{array}{ll}
                             r(x^k)-\frac{1}{2}(\frac{1}{\alpha}- L_r)\|x^{k}(\alpha)-x^{k}\|^2,&\alpha \in \left(0,\frac{1}{L_r}\right) \\
                             ~~&\\
                             r(x^k)-\frac{\sigma}{2} \frac{\|x^{k}(\alpha)-x^k\|^2}{\alpha^2},&\alpha \in \left[\frac{1-\sqrt{1-4\sigma L_r}}{2L_r},\frac{1+\sqrt{1-4\sigma L_r}}{2L_r}\right].
                           \end{array}
\right.
\end{equation}
\end{lemma}
\begin{proof}
From the algorithm in Step 2, we have $$x^{k}(\alpha)\in \mathrm{argmin}\left\{\|x-x^k+\alpha \nabla r(x^k) \|^2, \|x\|_0\leq s, x\geq 0 \right\},$$
which implies that $\|x^{k}(\alpha)-x^k+\alpha \nabla r(x^k) \|^2 \leq \| \alpha\nabla r(x^k) \|^2$, that is
\begin{equation}\label{eq8}
\|x^{k}(\alpha)-x^k\|^2\leq -2\alpha \langle \nabla r(x^k),x^{k}(\alpha)-x^k\rangle.
\end{equation}
From
\begin{eqnarray}
r(x^{k}(\alpha))&=&r(x^{k})+\langle \nabla r(x^{k}),x^{k}(\alpha)-x^{k} \rangle+\frac{1}{2}\|A(x^{k}(\alpha)-x^{k})\|^2\nonumber\\
       &\leq& r(x^{k})-\frac{1}{2\alpha}\|x^{k}(\alpha)-x^{k}\|^2+\frac{L_r}{2}\|(x^{k}(\alpha)-x^{k})\|^2
\end{eqnarray}
we can obtain the desired result by the definition of $\alpha$.
\qed\end{proof}

Using Lemma \ref{le3.1} and other properties of iterative sequence, the convergence properties of GSPA can be established.
\begin{theorem}\label{le3.2}
Let the sequence $\{x^k\}$ be generated by GSPA, we have \\
$\left(i\right)$~$\underset{k\rightarrow \infty}{\lim}\frac{\|x^{k+1}-x^{k}\|}{\alpha_k}=0$;\\
$\left(ii\right)$ any accumulation point of $\{x^k\}$ is the $\alpha$-stationary point of $(\ref{p});$\\
$\left(iii\right)$ $\lim_{k\rightarrow\infty}\|\nabla^C_{S\cap \mathbb{R}^N_+}r(x^k)\|=0$.
\end{theorem}
\begin{proof}
$(\romannumeral1)$ From (\ref{objective-descent-12}), we derive that $r(x^{k})-r(x^{k+1})\geq c \frac{\|x^{k+1}-x^k\|^2}{\alpha_{k}^2}$, where $c=\min \{\frac{\alpha_0-L_r\alpha_0^2}{2}, \frac{\sigma}{2}\}$.
Then
\begin{eqnarray*}
\sum^{\infty}_{k=0}\frac{\|x^{k+1}-x^k\|^2}{\alpha_{k}^2}\leq \frac{1}{c}\sum^{\infty}_{k=0}\left( r(x^{k})-r(x^{k+1})\right)= \frac{1}{c}r(x^{0})<+\infty,\end{eqnarray*}
 which signifies
$\underset{k\rightarrow \infty}{\lim}\frac{\|x^{k+1}-x^k\|}{\alpha_k}=0$. Since $\alpha_k$ is bounded from below by a positive constant, we conclude that $\underset{k\rightarrow \infty}{\lim}\|x^{k+1}-x^k\|=0.$

$(\romannumeral2)$ Suppose that $x^*$ is an accumulate point of the sequence $\{x^k\}$, then there exists a subsequence $\{x^{k_n}\}$ that
converges to $x^*$. By ($\romannumeral1)$, $\underset{n\rightarrow \infty}{\lim}x^{k_n+1}=x^*$. Based on $x^{k_n+1}=\textmd{P}_{S\cap \mathbb{R}_+^N}\big(x^{k_n}-\alpha \nabla r(x^{k_n})\big)$ in Step 2, we consider two cases.

Case 1. $i\in \mathrm{supp}(x^*)$. The convergence of $\{x^{k_n}\}$ and $\{x^{k_n+1}\}$ guarantees that for some $n_1>0$,
$x_i^{k_n}>0$, $x_i^{k_n+1}>0$ for all $n>n_1$. The definition of the projection on $S\cap \mathbb{R}_+^N$ shows that
$$ x_i^{k_n+1}=x_i^{k_n}-\alpha \nabla_i r(x^{k_n}).$$
Taking $n\rightarrow\infty$, we have $\nabla_i r(x^*)=0$.

Case 2. $i\notin \mathrm{supp}(x^*)$. If there exists an $n_2>0$ such that for all $n>n_2$, $x_i^{k_n+1}=0$, the projection implies that
$$
x_i^{k_n}-\alpha\nabla_i r(x^{k_n})\leq 0~ \mathrm{or}~~0<x_i^{k_n}-\alpha\nabla_i r(x^{k_n})\leq M_s(x^{k_n+1}).
$$
Letting $n\rightarrow\infty$ and exploiting the continuity of the function $M_s$, we obtain that
$$\nabla_i r(x^*)\geq 0~ \mathrm{or}~~-M_s(x^{*})\leq\alpha\nabla_i r(x^*)\leq 0.$$
On the other hand, if there exists an infinite number of indices of $k_n$ for $x_i^{k_n+1}>0$, as the same proof in Case 1, it follows that
$\nabla_i r(x^*)=0$.
Since $\alpha_k$ is bounded from below by a positive constant, we have
\begin{eqnarray}
 \nabla_i r(x^*)\left\{
             \begin{array}{lll}
              =0, & \mbox{if}~~ i\in \mathrm{supp}(x^*),\\
             \geq 0 , \mathrm{or} \in[-\frac{1}{\alpha}M_s (x^*),0], & \mbox{if}~~i \notin \mathrm{supp}(x^*);
             \end{array}
        \right.
\end{eqnarray}
which means $x^*$ is an $\alpha$-stationary point of (\ref{p}) by Theorem \ref{th2.1}.

$(\romannumeral3)$ 
From Theorem \ref{pro-tan-nor-SN}, $T^C_{S\cap\mathbb{R}^N_+}(x^k)=\{~d\in\mathbb{R}^{N}~|~\textup{supp}(d)\subseteq \textup{supp}(x^k)~\}$ is a subspace.
Then $$\|\nabla^C_{S\cap \mathbb{R}^N_+}r(x^k)\|=\max\{\langle -\nabla r(x^k), v^k\rangle|v^k\in T^C_{S\cap\mathbb{R}^N_+}(x^k),\|v^k\|\leq 1\},$$
see Lemma 3.1 in \cite{Calamai87}. we have for any $\eta>0$, there is a $v^k \in T^C_{S\cap\mathbb{R}^N_+}(x^k)$ with $\|v^k\|=1$ satisfying
\begin{equation}\label{pro-grad}
 \|\nabla_{S\cap\mathbb{R}^N_+}r(x^k)\|\leq -\langle \nabla r(x^k),v^k\rangle+\eta.\end{equation}
For all $z^{k+1}\in T^C_{S\cap\mathbb{R}^N_+}(x^{k+1})$ and $x^{k+1}=\textmd{P}_{S\cap \mathbb{R}_+^N}\big(x^k-\alpha_k \nabla r(x^k)\big)$,
$x^{k+1}-(x^k-\alpha_k \nabla r(x^k))$ is orthogonal to $T^C_{S\cap\mathbb{R}^N_+}(x^{k+1})$, which yields that
$$\langle x^{k+1}-(x^k-\alpha_k \nabla r(x^k)),z^{k+1}-x^{k+1}\rangle=0.$$
Letting $v^{k+1}=\frac{z^{k+1}-x^{k+1}}{\|z^{k+1}-x^{k+1}\|}\in T^C_{S\cap\mathbb{R}^N_+}(x^{k+1})$, with Cauchy-Schwartz inequality, the above equation leads to
\begin{equation}\label{delta-r}
-\langle\nabla r(x^k),v^{k+1}\rangle\leq \frac{\|x^{k+1}-x^k\|}{\alpha_k}.\end{equation}
From ($\romannumeral1)$, $\limsup_{k\rightarrow \infty}-\langle\nabla r(x^k),v^{k+1}\rangle\leq 0$.
Combining
$$
-\langle\nabla r(x^{k+1}),v^{k+1}\rangle=-\langle\nabla r(x^{k+1})-\nabla r(x^k),v^{k+1}\rangle-\langle \nabla r(x^k),v^{k+1}\rangle
$$
with (\ref{delta-r}), ($\romannumeral1)$ and Lipschitz continuity of $\nabla r(x)$, we have
$$\limsup_{k\rightarrow \infty}-\langle\nabla r(x^{k+1}),v^{k+1}\rangle\leq 0.$$
By (\ref{pro-grad}) and the arbitrariness of $\eta$, we can prove the result.
\qed
\end{proof}

In order to attain the result that $\{x^k\}$ converges to a local minimizer of (\ref{p}), we need the following assumption and lemma.
\begin{assumption}(\cite{Beck13})\label{s-regular}
Matrix $A$ is $s$-regular if any $s$ of its columns are linearly independent, namely,
\begin{eqnarray}\label{s-regular-N}d^{\top}A^{\top}Ad>0,~~~\forall~\|d\|_0\leq s.\end{eqnarray}
\end{assumption}
\begin{theorem}\label{th3.3}
Let the sequence $\{x^k\}$ be generated by GSPA, then $\{x^k\}$ converges to a local minimizer of $(\ref{p})$ if Assumption \ref{s-regular} holds.
\end{theorem}
\begin{proof}
For completeness, we give the proof similar to that of Theorem 3.2 in \cite{Beck13}.

First, the number of $\alpha$-stationary points of (\ref{p}) is finite.
In fact, by Theorem \ref{th2.1}, $\alpha$-stationary point $x^*$ satisfies $$\nabla_\Gamma r(x^*)=A_{\Gamma}^T(A_{\Gamma}x_{\Gamma}-b)=0,~\Gamma=\mathrm{supp}(x^*),~x_{\Gamma}\geq 0, ~|\Gamma|\leq s, $$
which has at most one solution. Since the number of subsets of $\{1,2,\cdots,N\}$ whose size is no larger than $s$ is $T=\sum_{i=0}^s C_N ^i$,
the number of $\alpha$-stationary points of (\ref{p}) is no more than $T$.

Now we show $\{x^k\}$ is bounded. Lemma \ref{le3.1} indicates that $\{r(x^k)\}$ is decreasing, then
 the sequence $\{x^k\}$ is contained in the level set
$$ E=\{~x\in \mathbb{R}^N_+\cap S~|~r(x)\leq r(x^0)~\}. $$
We can represent the set $E$ as the union $E=\bigcup _{j=1}^T E_j$, where
$E_j=\big\{x\in\mathbb{R}^N_+|\|Ax-b\|^2\leq r(x^0), x_i=0, i\notin\Gamma_j, ~|\Gamma_j|\leq s \big\}.$
Since Assumption \ref{s-regular} implies that $A^\Gamma_{\Gamma_j}A_{\Gamma_j}$ is positive definite, $E_j$ is bounded, which shows $E$ is bounded.
Combining the boundedness of $\{x^k\}$ and $(\romannumeral2)$ in Theorem \ref{le3.2}, we obtain that there is a  subsequence $\{x^{k_n}\}$ converges to an $\alpha$-stationary point $x^*$.

We conclude that $\lim_{k\rightarrow \infty}x^k=x^*$.
Since the number of $\alpha$-stationary points of (\ref{p}) is finite,
there exists an $\epsilon>0$ smaller than the minimal distance between all the pairs of the $\alpha$-stationary points.
We show the convergence of $x^k$ by contradiction. When $n$ is sufficiently large,
$\|x^{k_n}-x^*\|\leq \epsilon$, without loss of generality, we assume the above inequality holds for
all $n\geq 0$.
Since $x^k$ is divergent, the index $l_n$ given by
$$
l_n=\min\{~i~|~\|x^i-x^*\|> \epsilon,i>k_n,~i\in \mathbb{N}~\}
$$
is well defined. We have thus constructed a subsequence $\{x^{l_n}\}$ for which
$$
\|x^{l_n-1}-x^*\|\leq \epsilon, ~\|x^{l_n}-x^*\|>\epsilon, n=1,2,\cdots
$$
It follows that $\{x^{l_n-1}\}$ converges to $x^*$, there exists an $n_0>0$ such that for all $n>n_0$,
$\|x^{l_n-1}-x^*\|\leq \frac{\epsilon}{2}$. Then for all $n>n_0$, $\|x^{l_n-1}-x^{l_n}\|>\frac{\epsilon}{2}$,
contradicting $(\romannumeral1)$ in Theorem \ref{le3.2}.

Finally, by $s$-regularity of $A$ and Theorem \ref{B-Sufficient-N}, it has that $x^*$ is also the local minimizer of $(\ref{p})$. Therefore $\{x^k\}$ converges to a local minimizer of $(\ref{p})$.\qed
\end{proof}

By analyzing the convergence theorems, we can obtain the theorem of existence of optimal solution of (\ref{p}),
which can be regarded as the theorem of second order sufficient optimality condition.

 \begin{theorem}\label{th4.3}
Suppose that Assumption \ref{s-regular} holds for matrix $A$, then the local solutions of problem (\ref{p}) exist and are finite. Moreover, its global solution exists consequently.
\end{theorem}
\textbf{Remark } We achieve the stronger convergence results (Theorem \ref{le3.2}, \ref{th3.3} and \ref{th4.3})
under relatively weaker assumption compared with \cite{Attouch13,Blumensath09}.
More exactly, the gradient projection algorithm in \cite{Attouch13} converges to a $N$-stationary point provided the iteration sequence is bounded,
while GSPA has the same result without boundedness of the iterative sequence.
NIHT in \cite{Blumensath09} converges to a local minimizer of (\ref{p}) if $A$ satisfies the
 restricted isometry property (RIP) (introduced in \cite{Candes05}), while GSPA has the same convergence result with $s$-regularity of $A$, which is weaker than RIP.

\section{Numerical Experiments}

Before proceeding to the computational results, we need to define some notations and data sets.
For convenience and clear understanding in the graph presentations and some comments, we use the notations: $GSPA$, $NIHT$, $CSMP$ (short for $CoSaMP$) and $SP$ to represent the our Gradient Support Projection Algorithm, Normalized Iterative Hard Thresholding (proposed by Blumensath in \cite{Blumensath09}), Compressive Sampling Matching Pursuit (established by Thomas et al. in \cite{Needell09} ), and Subspace Pursuit in \cite{Dai09}  respectively without nonnegative constraints. We first will compare the numerical performance of $GSPA$ and $NIHT$ under the nonnegative constraints, and thus denote them as $N\_GSPA$ and $N\_NIHT$.

To accelerate the rate of convergence, $\alpha_0$ in each iteration is chosen according to \cite{Blumensath09}
$$
\alpha_0^k=\frac{\|A^T_{\Gamma^k}(b-Ax^k)\|^2}{\|A_{\Gamma^k}A^T_{\Gamma^k}(b-Ax^k)\|^2},$$
For each data set, the random matrix $A$ and the designed vector $b$ in absence of nonnegative constraints are
generated by the following MATLAB codes:
\begin{eqnarray}&&x_{\textmd{orig}} = \textmd{zeros}(N, 1),~~y = \textmd{randperm}(N),~~x_{\textmd{orig}}(y(1 : s)) = \textmd{randn}(s, 1),\nonumber\\
&&A = \textmd{randn}(M, N),~~~b = A*x_{\textmd{orig}}. \nonumber\end{eqnarray}
If considering the nonnegative constraints, we simply alter corresponding code as
$$x_{\textmd{orig}}(y(1 : s)) = \textmd{abs}(\textmd{randn}(s, 1)).$$
where the sparsity $s$ is  taking $s=1\%N$ or $k=5\%N$. In terms of parameters, we fix $\beta=0.8$ and $\sigma=10^{-5}$ for simplicity. For each data set, we will randomly run $40$ samples and the stopping criterias will be set by $\|x^{k+1}-x^{k}\|\leq10^{-6}$ or the maximum iterative times is equal to 5000 for all methods. In the following analysis, we say $x$ as the recovered solution from the affine equations. In whole experiments, the average prediction error $\|Ax-b\|$, the recovered error $\|x-x_{\textmd{orig}}\|_\infty$ and CPU time will be taken into consideration to illustrate the performance of the four methods. All those simulations are carried out on a CPU 2.6GHz laptop.

\subsection{Comparison of $N\_NIHT$ and $N\_GSPA$}
We first will compare $N\_NIHT$ (which can be simply altered by adding the nonnegative constraints $x\geq0$ when pursuing the projection in $NIHT$) and $N\_GSPA$ by setting different $N$ with $s=5\%N$ and running $40$ samples for each data set.
\vspace{3mm}

 \makeatletter
          \def\@captype{figure}
          \makeatother
         {\centering \includegraphics[width=16.2cm]{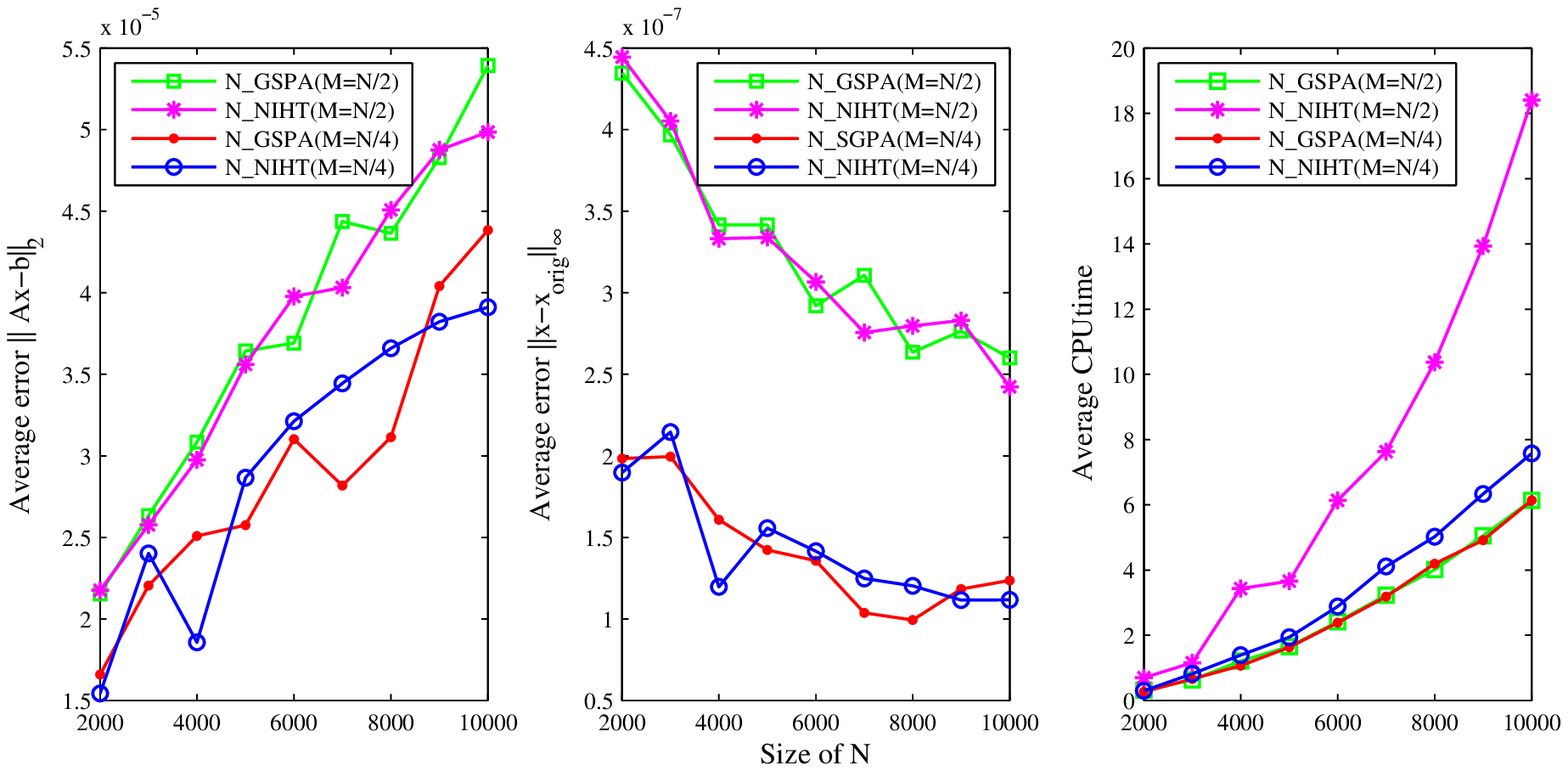}}
          \vspace{-5mm}
          \caption{Average results yielded by $N\_NIHT$  and $N\_GSPA$ under $M=N/2$ and $M=N/4$.\label{fige1}}

\vspace{3mm}

 The corresponding results can be seen in Figure \ref{fige1}, from which one can discern that when $M=N/2$ and $M=N/4$ there is no big difference of these two methods. To be more exact, the average prediction error $\|Ax-b\|$ and recovered error $\|x-x_{\textmd{orig}}\|_\infty$ are quite small with order of magnitude $10^{-5}$ and $10^{-7}$, which means the recover is almost exact. In terms of the average CPU time, see the third column of Figure \ref{fige1}. For one thing, the time cost by $N\_GSPA$ is smaller than that form $N\_NIHT$ for each case. For another, one can check that $N\_NIHT$ are relatively dependent on the sample dimension $M$. The larger the $M$ is, the smaller is the time, which implies  the performance of $N\_NIHT$ tends to be worse with the $M$ decreasing for fixed $N$. By contrast, the performance of our $N\_GSPA$ approach is much more stable since the time cost by this method is nearly similar under different $M$.
\subsection{Comparison of $GSPA$, $NIHT$, $CSMP$ and $SP$}
In the sequent part, we mainly compare $GSPA$, $NIHT$, $CSMP$ and $SP$ without the nonnegative constraints. The primal MATLAB codes of $CSMP$ and $SP$ can be download from the website below:

\noindent http://media.aau.dk/null\_space\_pursuits/2011/07/a-few-corrections-to-cosamp-and-sp-matlab.html.

\textbf{Exact Recovery:}~~We firstly consider the exact sparse recovery $b = Ax_{\textmd{orig}}$. Through running 40 examples, the produced data is listed as Tables \ref{tab1}--\ref{tab3}.

 \tabcaption{The average prediction error $\|Ax-b\|$ over 40 simulations with $s=5\%N$. \label{tab1}}
{\renewcommand\baselinestretch{1.15}\selectfont
{\centering\begin{tabular}{ c c c c c c }\hline
~~~~~$N$~~~~~&~~~~~$M$~~~~~&~~~~~$GSPA$~~~~~&~~~~~$NIHT$~~~~~&~~~~~$CSMP$~~~~~&~~~~~$SP$~~~~~\\\hline
\multirow{2}{*}{~~~$N=1000$}&$M=N/4$&0.14e-04&    0.15e-04&    0.00e-04&    0.00e-04~~~\\
& $M=N/2$&0.09e-04&    0.09e-04&    0.00e-04&    0.00e-04~~~\\\hline
\multirow{2}{*}{~~~$N=3000$} &$M=N/4$&0.29e-04&    0.27e-04&    0.00e-04&    0.00e-04~~~\\
&$M=N/2$&0.17e-04&    0.20e-04&    0.00e-04&    0.00e-04~~~\\\hline
\multirow{2}{*}{~~~$N=5000$} &$M=N/4$ &   0.36e-04&    0.34e-04&    0.00e-04&    0.00e-04~~~\\
&$M=N/2$&    0.23e-04&    0.24e-04&    0.00e-04&    0.00e-04~~~\\\hline
 \multirow{2}{*}{~~~$N=7000$}&$M=N/4$&   0.42e-04&    0.41e-04&    0.00e-04&    0.00e-04~~~\\
  &$M=N/2$&  0.29e-04&    0.32e-04&    0.00e-04&    0.00e-04~~~\\\hline
\multirow{2}{*}{~~~$N=10000$} &$M=N/4$&    0.51e-04&    0.48e-04&    0.00e-04&    0.00e-04~~~\\
 &$M=N/2$&   0.37e-04&    0.39e-04&    0.00e-04&    0.00e-04~~~\\\hline
\end{tabular}\par}}
\vspace{6mm}

 \tabcaption{The average recovered error $\|x_{\textmd{orig}}-x\|_\infty$ over 40 simulations with $s=5\%N$. \label{tab2}}
{\renewcommand\baselinestretch{1.15}\selectfont
{\centering\begin{tabular}{ c c c c c c }\hline
~~~~~$N$~~~~~&~~~~~$M$~~~~~&~~~~~$GSPA$~~~~~&~~~~~$NIHT$~~~~~&~~~~~$CSMP$~~~~~&~~~~~$SP$~~~~~\\\hline
\multirow{2}{*}{~~~$N=1000$}&$M=N/4$&        0.50e-06&    0.51e-06&    0.00e-06&    0.00e-06~~~\\
  & $M=N/2$&  0.20e-06&    0.16e-06&    0.00e-06&    0.00e-06~~~\\\hline
  \multirow{2}{*}{~~~$N=3000$} &$M=N/4$&  0.45e-06&    0.38e-06&    0.00e-06&    0.00e-06~~~\\
   &$M=N/2$& 0.16e-06&    0.17e-06&    0.00e-06&    0.00e-06~~~\\\hline
 \multirow{2}{*}{~~~$N=5000$}  &$M=N/4$&   0.31e-06&    0.30e-06&    0.00e-06&    0.00e-06~~~\\
 &$M=N/2$&   0.12e-06&    0.11e-06&    0.00e-06&    0.00e-06~~~\\\hline
   \multirow{2}{*}{~~~$N=7000$} &$M=N/4$&    0.26e-06&    0.26e-06&    0.00e-06&    0.00e-06~~~\\
   & $M=N/2$&    0.12e-06&    0.12e-06&    0.00e-06&    0.00e-06~~~\\\hline
 \multirow{2}{*}{~~~$N=10000$} &$M=N/4$&   0.24e-06&    0.24e-06&    0.00e-06&    0.00e-06~~~\\
    &$M=N/2$&    0.10e-06&    0.10e-06&    0.00e-06&    0.00e-06~~~\\\hline
\end{tabular}\par}}
\vspace{4mm}

 \tabcaption{The average CPU time over 40 simulations with $s=5\%N$. \label{tab3}}
{\renewcommand\baselinestretch{1.15}\selectfont
{\centering\begin{tabular}{ c c c c c c }\hline
~~~~~$N$~~~~~&~~~~~$M$~~~~~&~~~~~$GSPA$~~~~~&~~~~~$NIHT$~~~~~&~~~~~$CSMP$~~~~~&~~~~~$SP$~~~~~\\\hline
\multirow{2}{*}{~~~$N=1000$}&~~$M=N/4$~~&    0.0689~~&    0.2583~~&    0.1492~~&    0.0961~~~\\
 & $M=N/2$&     0.0677~~&    0.2459~~&    0.1687~~&    0.1307~~~\\\hline
 \multirow{2}{*}{~~~$N=3000$} &~~$M=N/4$~~&     0.5385~~&    3.3210~~&    1.9171~~&    1.1197~~~\\
   &~~$M=N/2~~$&  0.5756~~&    2.6228~~&    1.8754~~&    1.3627~~~\\\hline
 \multirow{2}{*}{~~~$N=5000$}  &~~$M=N/4$~~&   1.5583~~&    11.246~~&    8.0507~~&    4.5900~~~\\
&$~~M=N/2~~$&     1.5114~~&    8.0690~~&    7.7457~~&    5.0981~~~\\\hline
   \multirow{2}{*}{~~~$N=7000$} &~~$M=N/4$~~&    3.0050~~&    20.761~~&    19.698~~&    10.729~~~\\
  & $~~M=N/2~~$&  2.9543~~&    16.389~~&    19.336~~&    12.613~~~\\\hline
   \multirow{2}{*}{~~$N=10000$} &~~$M=N/4$~~&    6.3880~~&    52.257~~&    51.680~~&    27.864~~~\\
    &~~$M=N/2$~~&     5.9462~~&   38.256~~&    53.707~~&    30.924~~~\\\hline
\end{tabular}\par}}
\vspace{5mm}
From Tables \ref{tab1} and \ref{tab2}, although the errors of $\|Ax-b\|$ and $\|x_{\textmd{orig}}-x\|_\infty$ resulted from $CSMP$ and $SP$ are basically equal to zero, the others stemmed from $GSPA$ and $NIHT$ are approximately close to zero as well, and thus there is no big distinction of recovered effects among those four methods. However, one can not be difficult to find that in Table \ref{tab3} the average CPU time cost by $GSPA$ is much lower than those spent by three other methods, which means under such circumstance our proposed approach run extremely fast. For instance when $N=10000$ with $M=N/2~(M=N/4$) and $s=5\%N$, the CPU time only need 5.9462(6.3880) seconds via $GSPA$, while 38.256(52.257), 53.707(51.680), 30.924(27.864) seconds yielded by $NIHT$, $CSMP$ and $SP$ respectively.
\vspace{3mm}

  \makeatletter
          \def\@captype{figure}
          \makeatother
{\centering \includegraphics[width=16.cm]{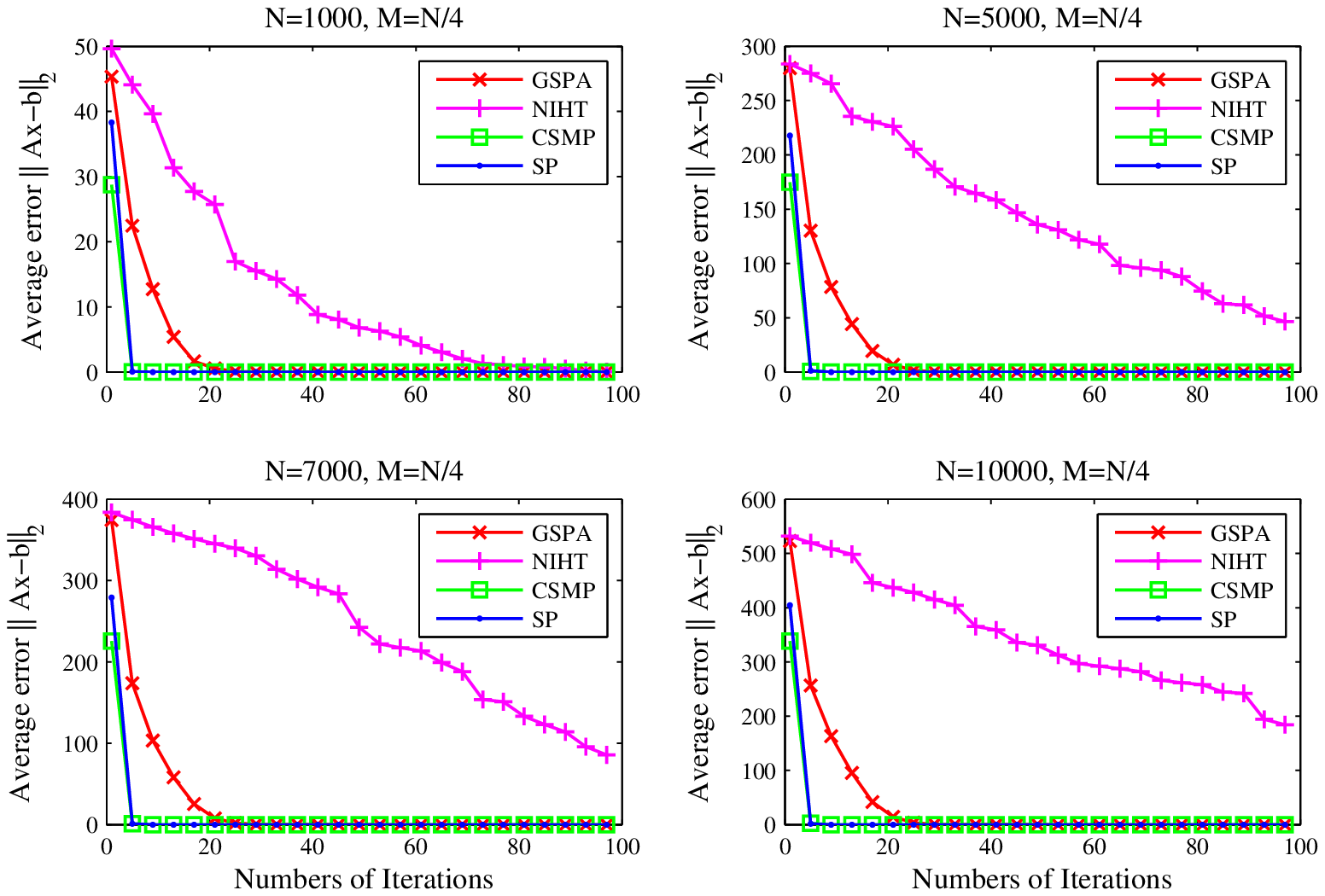}}
\vspace{-5mm}
          \caption{Average prediction error $\|Ax-b\|$ for each iteration with $s=5\%N$ over $40$ simulations.\label{fige5}}
\vspace{2mm}

We then run $40$ simulations to count the average error $\|Ax-b\|$ for each iteration. The first 100 iterations will be taken into account to observe the average descent rate of the error $\|Ax-b\|$ from four algorithms. Seeing Figure \ref{fige5}, $NIHT$ requires the far over 100 iterative times to make the error $\|Ax-b\|$ decline to a desirable level. Compared with that, $GSPA$ almost need 20 iterative times for any dimensions $N$ to reach the lowest level. Even though times of iterations (nearly 7 times for each $N$ ) demanded by $CSMP$ and $SP$ are smallest among these four algorithms, it also indicates that each iteration must cost a relatively long time based on the CPU time in Table \ref{tab3}.
\vspace{1mm}

\textbf{Recovery with Noise:}~~For the sake of clarity of illustrating the robust of these algorithms,  we sequently simulate the recovery with noisy case. Corresponding MATLAB codes are modified to:
\begin{eqnarray}&&x_{\textmd{orig}} = \textmd{zeros}(N, 1),~~y = \textmd{randperm}(N),~~x_{\textmd{orig}}(y(1 : s)) = \textmd{randn}(s, 1),\nonumber\\
&&A = \textmd{randn}(M, N),~~~b = A*x_{\textmd{orig}}+\sigma_0* \textmd{randn}(M,1), \nonumber\end{eqnarray}
where the noise obeys to the normal distribution with zero expectation and $\sigma_0^2$ (taken as $\sigma_0=0.01$ for simplicity) variance. Specific figures produced by these four approaches when $M=N/4$ and $k=5\%N$ are recorded in Table \ref{tab4}, where "-- --" denotes the invalid computation. The most significant property of the data in the table is the recovered effects ($\|Ax-b\|$ or $\|x_{\textmd{orig}}-x\|_\infty$) of $GSPA$, $NIHT$, $CSMP$ and $SP$ are almost nondistinctive. In other words, with noise disturbing, $CSMP$ and $SP$ no longer perform as well as that in absence of noise. Particularly, when the sample size $N\geq5000$, $CSMP$ behaves extremely worse so that it is impossibly implementary in the high dimensional real applications. What makes the results stunning in Table \ref{tab4} is that the average CPU time of $GSPA$ is far of smallness, comparing with time spent by $NIHT$, $CSMP$ and $SP$, which indicates these three methods are not appealing when the affine equations are interfered by some noise, even though the noise is quite minute.

 \tabcaption{Average results over 40 simulations with $M=N/4, s=5\%N$ and noise. \label{tab4}}
{\renewcommand\baselinestretch{1.15}\selectfont
{\centering\begin{tabular}{ c c c c c c }\hline
&~~~~~~~~~$N$~~~~~~~~~&~~~~~~$GSPA$~~~~~~&~~~~~~$NIHT$~~~~~~&~~~~~~$CSMP$~~~~~~&~~~~~~$SP$~~~~~~\\\hline
\multirow{5}{*}{$\|Ax-b\|$} &1000~~&    0.1376&    0.1376&    0.1723 &   0.1376~~~\\
 &3000~~&   0.2505 &   0.2505&    0.3056 &   0.2505~~~\\
 &5000~~&   0.3216&    0.3216&    -- --     &   0.3216~~~\\
  &7000~~&  0.3718&    0.3716&    -- --    &   0.3725~~~\\
  &10000~~&  0.4478&    0.4478&    -- --    &   0.4487~~~\\\hline
\multirow{5}{*}{$\|x_{\textmd{orig}}-x\|_\infty$}&1000~~&    0.0022&     0.0022&     0.0032 &    0.0022~~~\\
&3000~~&    0.0012&     0.0012&     0.0020 &    0.0012~~~ \\
&5000~~&    0.0010&     0.0010&      -- --    &       0.0010~~~\\
&7000~~&    0.0007&     0.0008&      -- --     &  0.0011~~~\\
&10000~~&    0.0008&     0.0008&       -- --    &  0.0009~~~\\\hline
 \multirow{5}{*}{CPU time}&1000~~&   0.0812&    0.3226&    116.87  &  0.1859~~~\\
 &3000~~&   0.5797&    3.9317&    1416.1  &  1.1631~~~\\
  &5000~~&  1.6221&    9.6857&     -- --     &  4.9076~~~\\
 &7000~~&   3.2252&    25.306&    -- --      &  11.556~~~\\
  &10000~~&  6.6369&    38.440&     -- --      &  28.429~~~\\\hline
    \end{tabular}\par}}
\vspace{4mm}

In Figure \ref{fige6}, for the comparison between $GSPA$ and $NIHT$, one can check that the average prediction error $\|Ax-b\|$ begins to close to zero when $GSPA$ iterates nearly 20 steps, which is smaller than that $NIHT$ does. When it comes to compare $GSPA$ and $CSMP$, we reduce the sample size due to the time complexity of $CSMP$ (see Table \ref{tab4}). Although at the beginning the error of $CSMP$  descends dramatically (here 1-10 iterative times has not been plotted in middle of Figure \ref{fige6}), then it almost stabilizes at a small error  and  does not decline to zero again. By contrast, the error from $GSPA$ always drops until to zero. In terms of comparison between $GSPA$ and $SP$, one can observe the iterative times (approximately 7 times) for $SP$ to reach the bottom are relatively small, whilst the error from $GSPA$ requires nearly 10 (30) times when $M=N/2(M=N/4)$ to reach the lowest point. However, meticulous readers are not difficult to find that based on the CPU time in Table \ref{tab4}, the time for each iteration of $SP$ must cost longer than $GSPA$.
\vspace{5mm}

 \makeatletter
\def\@captype{figure}
\makeatother
{\centering \includegraphics[width=15.75cm]{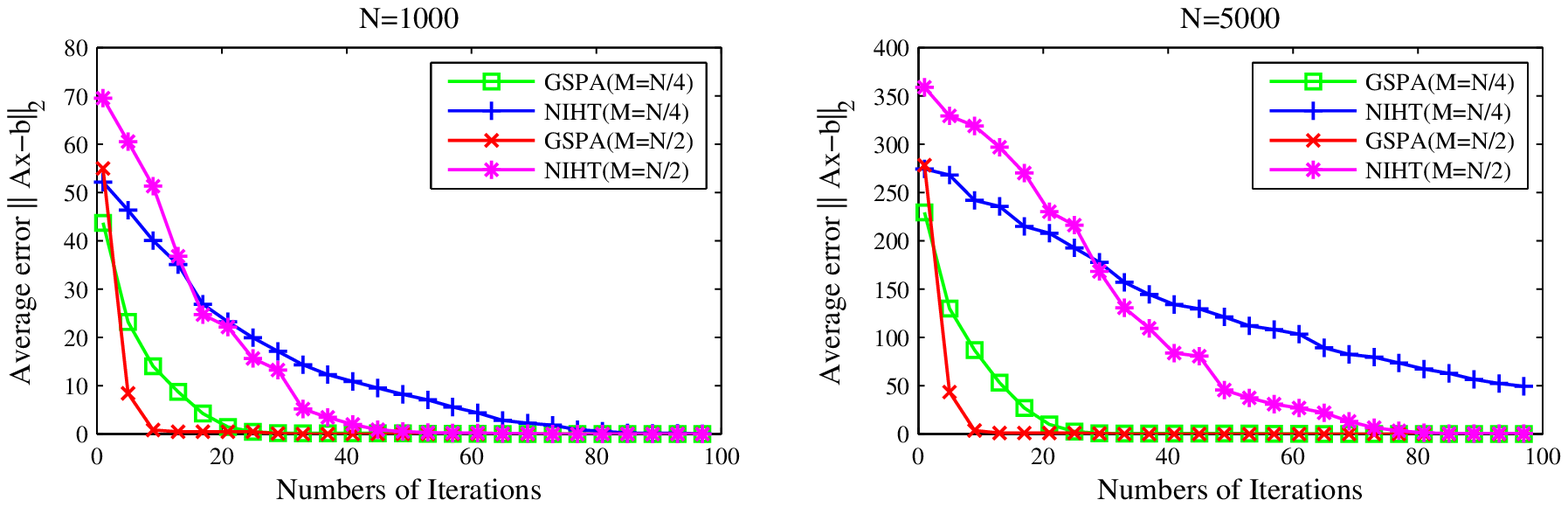}}

 \makeatletter
\def\@captype{figure}
\makeatother
{\centering \includegraphics[width=15.75cm]{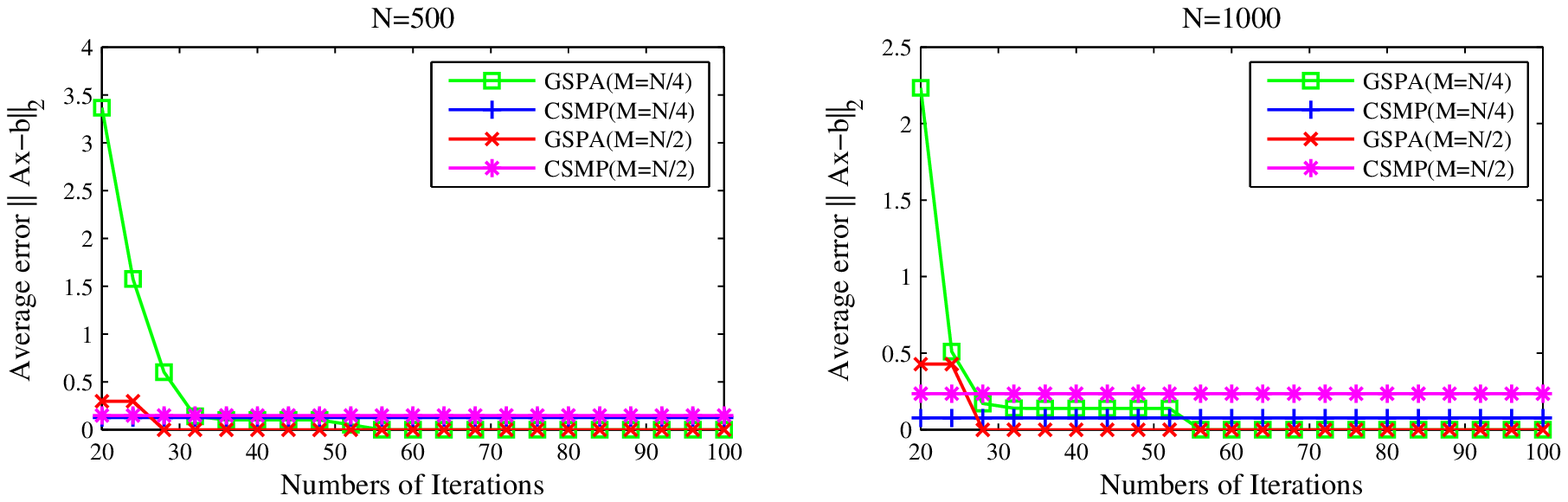}}

 \makeatletter
\def\@captype{figure}
\makeatother
{\centering \includegraphics[width=15.75cm]{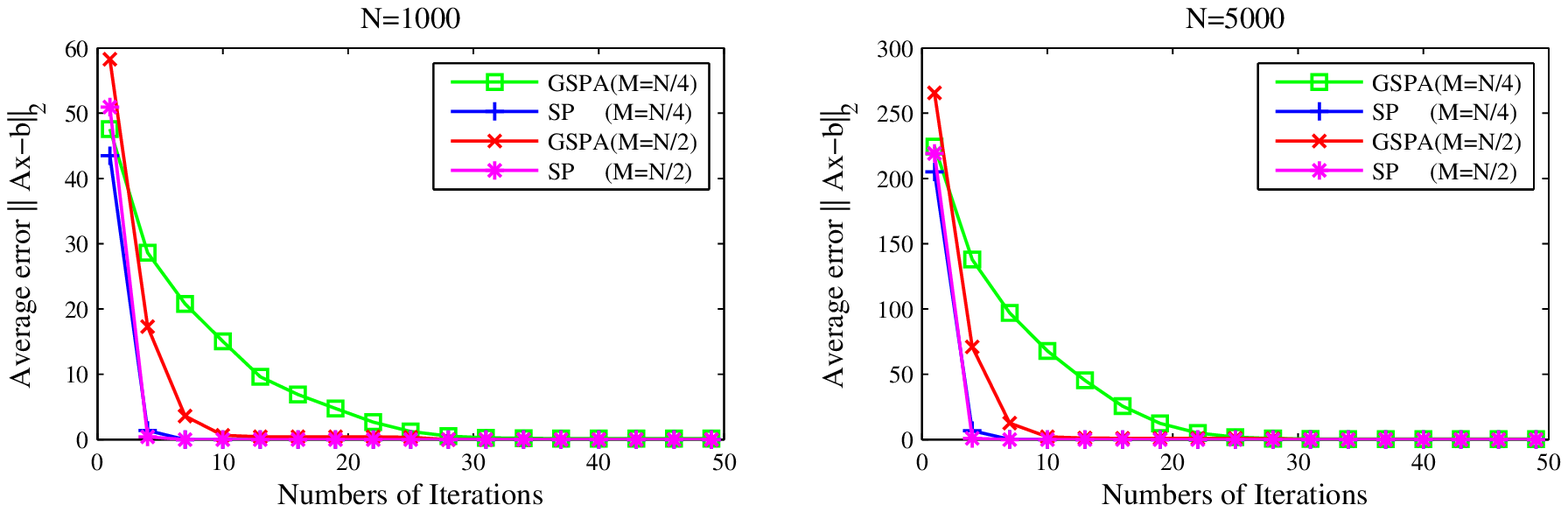}}
 \vspace{-1mm}
          \caption{Average error $\|Ax-b\|$ for each iteration with $s=5\%N$ over $40$ simulations with noise.\label{fige6}}
 \vspace{3mm}

Since the fact that $CSMP$ and $SP$ would perform worse under the relatively larger sparsity of $x_{\textmd{orig}}$, we consider the sparsity $s$ as $s=1\%N$ under the noisy case. The information in Table \ref{tab5} shows that when the sparsity $s$ of $x_{\textmd{orig}}$ is far less than $N$ ($s=1\%N$), $CSMP$ and $SP$ will perform as robustly as $GSPA$ and $NIHT$ do, because the corresponding results in Table \ref{tab5} of these four methods basically tend to be similar.

\vspace{2mm}

 \tabcaption{Average results over 40 simulations with $M=N/4, s=1\%N$ and noise. \label{tab5}}
{\renewcommand\baselinestretch{1.15}\selectfont
{\centering\begin{tabular}{  c c c c c c }\hline
&~~~~~~~~~$N$~~~~~~~~~&~~~~~~$GSPA$~~~~~~&~~~~~~$NIHT$~~~~~~&~~~~~~$CSMP$~~~~~~&~~~~~~$SP$~~~~~~\\\hline
\multirow{5}{*}{$\|Ax-b\|$} &1000&   0.1549  &  0.1548   & 0.1565   & 0.1549~~~\\
    &3000&0.2752  &  0.2752   & 0.2782  &  0.2752~~~\\
    &5000&0.3418  &  0.3418   & 0.3455  &  0.3418~~~\\
    &7000&0.4073  &  0.4073   & 0.4129  &  0.4073~~~\\
    &10000&0.4889 &   0.4889   & 0.4942 &   0.4889~~~\\\hline
\multirow{5}{*}{$\|x_{\textmd{orig}}-x\|_\infty$}&1000& 0.0011 &   0.0011 &   0.0015   & 0.0012~~~\\
    &3000&0.0007&    0.0007 &   0.0010 &   0.0007~~~\\
    &5000&0.0007 &   0.0007 &   0.0009 &   0.0007~~~\\
    &7000&0.0007&    0.0007 &   0.0008 &   0.0007~~~\\
    &10000&0.0006  &  0.0006&    0.0007  &  0.0006~~~\\\hline
\multirow{5}{*}{CPU time}&1000&0.0232  &  0.0543  &  0.0455 &   0.0134~~~\\
    &3000&0.1368 &   0.3334 &   0.1559 &   0.0739~~~\\
    &5000&0.3383&   1.1299 &   0.5412 &   0.1848~~~\\
    &7000&0.6440 &   1.9878 &   1.6966 &   0.5325~~~\\
    &10000& 1.4000&    4.5121&    3.3096  &  1.1757~~~\\\hline
\end{tabular}\par}}
\vspace{4mm}

\subsection{Comments}

 From these two comparisons: comparison of $N\_NIHT$ and $N\_GSPA$ and  comparison of $GSPA$, $NIHT$, $CSMP$ and $SP$, some comments can be concluded.
\begin{itemize}
  \item There is no essential distinction between our $N\_GSPA$ and $GSPA$, because the projection on a nonnegative cone does not obstruct the computational time and recovered effects. From experiments and analysis above, the proposed method $GSPA$ performs very steadily, and thus does not be overly relied on the sample size $M$ and $N$. It also runs relatively well for some different sparsity $s$ of $x_{\textmd{orig}}$. In addition, regardless of the exact recovery and case with noise, $GSPA$ unravels its good robustness. Importantly, $GSPA$ is the most fast of all these four approaches;
  \item For exact recovery, $NIHT$, $CSMP$ and $SP$ all proceed a good performance, particularly the two latter approaches enable the recovery to be exceptionally exact (i.e., making the error $\|Ax-b\|$ and $\|x_{\textmd{orig}}-x\|_\infty$ extremely equal to zero), but the recovered effect of $NIHT$ quite depends on the sample size $M$ and $N$. When referring to the recovery with noise, the recovered effects from $CSMP$ and $SP$ are no longer better than $GSPA$ and $NIHT$, particularly the performance of $CSMP$ which excessively relies on the sparsity $s$ are becoming much worse. Moreover, the CPU time generated by these three methods is all far higher than that needed by $GSPA$, which implies in high dimensional recovery they would not be appealing.
\end{itemize}

\section{Conclusion}

In this paper, we have established the first and second order optimality conditions for model (\ref{p}) and (\ref{f-l-0}),
proposed a gradient support projection algorithm for $AFP_{SN}$, and  shown that the new algorithm has elegant convergence and exceptional performance.
In the future, we will develop this algorithm for solving splitting feasibility problem (by Censor in \cite{Censor94}) with sparsity and other complex constraints.

\section*{Acknowledgements}

We are grateful to Dr. Caihua Chen in Nanjing University for his helpful advice. The work was supported in part by the National Basic Research
Program of China (2010CB732501), and the National
Natural Science Foundation of China (11171018, 71271021).

\end{document}